\documentclass[11pt, reqno]{amsart}

\usepackage[margin=1.0in]{geometry}
\usepackage[english]{babel}
\usepackage[utf8]{inputenc}
\usepackage{amsmath,amssymb,amsthm, comment, mathtools}
\usepackage{hyperref,mathrsfs}

\newcommand{\R}{\mathbb{R}}

\newcommand{\N}{\mathcal{N}} 
\newcommand{\n}{n} 

\newcommand{\E}{\mathcal{E}}
\newcommand{\A}{\mathcal{A}}
\newcommand{\B}{\mathcal{B}}

\newcommand{\K}{K}

\newcommand{\F}{\mathcal{F}}

\DeclareMathOperator{\tr}{tr}

\newcommand{\D}{\mathcal{D}}

\newcommand{\pa}{\partial}
\newcommand{\ve}{\varepsilon}

\newcommand{\dist}{\mathrm{dist}}

\newcommand{\Dtb}{\hat{D}_t}

\newcommand{\ts}{\mathscr{E}}

\renewcommand{\Re}{\operatorname{Re}}
\renewcommand{\Im}{\operatorname{Im}}

\newcommand{\na}{D} 
\newcommand{\tn}{\overline{\na}\,}
\newcommand{\wt}{\beta} 

\let\div\relax 

\DeclareMathOperator{\curl}{curl}
\DeclareMathOperator{\div}{div}

\numberwithin{equation}{section}

\newtheorem{theorem}{Theorem}
\numberwithin{theorem}{section}
\newtheorem{prop}{Proposition}
\numberwithin{prop}{section}
\newtheorem{cor}{Corollary}
\numberwithin{cor}{section}
\newtheorem{lemma}{Lemma}
\numberwithin{lemma}{section}
\theoremstyle{remark}
\newtheorem*{remark}{Remark}
\theoremstyle{definition}
\newtheorem{defi}{Definition}

\title{On the lifespan of three-dimensional
gravity water waves with vorticity}
\author{Daniel Ginsberg}
\address{Department of Mathematics, Johns Hopkins University, Baltimore MD 21218}
\email{dginsbe5@math.jhu.edu}
\begin{document}

\begin{abstract}
  We prove a long-term regularity result for three-dimensional gravity
water waves with small initial data but nonzero initial vorticity. We consider
solutions whose vorticity vanishes
on the free boundary and use this to derive a
system for the evolution of the free boundary which reduces to the
Zakharov/Craig-Sulem formulation in the irrotational case. We are able to
continue the solution until a time determined by the size of the initial
vorticity in such a way that if the vorticity is zero, one recovers
a lifespan $T\sim \ve^{-N}$ where $N$ can be taken arbitrarily large if the
initial data is taken to be arbitrarily smooth.
\end{abstract}

\maketitle

\section{Introduction}
The motion of an inviscid incompressible fluid occupying a region
$\D = \cup_{0 \leq t \leq T} \{ t \} \times \D_t$, $\D_t \subset \R^3$
is described
by the fluid velocity $v = (v_1,v_2, v_3)$
and a non-negative function $p$ known as the pressure.
If the fluid body is subject to the force of gravity, then
the equations of motion are given by Euler's equations:
\begin{align}
 \big(\pa_t + v^k \pa_k ) v_i &= -\pa_i p - e_3  \text{ in } \D_t,
 \quad
 \text{ where } \pa_i = \frac{\pa}{\pa x^i},
 \label{mom}
\end{align}
and conservation of mass:
\begin{align}
 \div v = \pa_i v^i & = 0, \quad \textrm{ in } \D_t.
 \label{mass}
\end{align}
Here, we are using the Einstein summation
convention and summing over repeated upper and lower indices and
writing $v^i = \delta^{ij}v_j$.
We have also chosen units so that the acceleration due to gravity is one
and are writing $e_3 = (0,0,1)$.
Fluid particles on the boundary move with the velocity of the fluid,
so that:
\begin{align}
v\cdot \n = \kappa,
 \label{freebdy}
\end{align}
where $\kappa$ is the normal velocity of $\pa \D_t$ and $\n$ is the unit
normal to $\pa \D_t$.
We assume that $\D_t$ is given by
$\D_t = \{ (x_1, x_2, y) : x_1, x_2 \in \R^2, y \leq h(t,x_1,x_2)\}$ for some function $h$,
 in which case
\eqref{freebdy} can be re-written as:
\begin{align}
 \pa_t h + v^1 \pa_1 h + v^2 \pa_2 h = v^3
 \quad \text{ on } \pa \D_t.
 \label{freebdy2}
\end{align}

If the fluid body moves in vaccuum and there is no surface tension on the
boundary then the pressure satisfies:
\begin{align}
 p = 0 \textrm{ on } \pa \D_t.
 \label{vaccuum}
\end{align}

Given $h_0:\R^2 \to \R$, set $\D_0 = \{(x_1, x_2, y) | y \leq h_0(x_1, x_2)\}$.
If $v_0 : \D_0 \to \R^3$ is a vector field satisfying the constraint
$\div v_0 = 0$, we want
to find a function $h$ and a vector field $v$ so that with $\D_t = \{
(x_1, x_2, y) | y \leq h(t,x_1, x_2)\}$, $v$ satisfyies \eqref{mom}-\eqref{mass}
and the initial conditions:
\begin{align}
 h(0, x_1, x_2) = h_0(x_1, x_2), \quad v = v_0 \quad \textrm{ on } \{0\}\times \D_0.
 \label{initial}
\end{align}
This problem is ill-posed
unless the following ``Taylor sign condition'' holds (see \cite{Ebin1987}):
\begin{align}
-\nabla_\n p(x,t) \geq \delta_0 > 0 \textrm{ on } \pa \D_t,
\quad \textrm{ where } \nabla_\n = \n^i \nabla_i,
\label{tsc}
\end{align}
where $\n$ denotes the unit normal to $\pa \D_t$.
This condition ensures that the pressure is positive in the interior of the fluid
and prevents the Rayleigh-Taylor instability from occuring.

In the irrotational
case ($\omega \equiv \curl v = 0)$, the velocity $v$ is given by $v = \nabla \psi$ for a
harmonic function $\psi : \D_t \to \R$, and the motion of the fluid is determined
entirely by $h$ and $\varphi = \psi\big|_{\pa \D_t}$. This, and related problems
have been studied extensively by several authors
in the case that the fluid domain $\D_t$ is diffeomorphic to the
half-space. See
for example \cite{Deng2017}, \cite{Ionescu2014},
\cite{Wu2009}, \cite{Germain2014}, as well as \cite{Ionescu2018}
 for a recent overview
of these problems. Let us single out
the works \cite{Wu2010}, \cite{Germain2012}, in which the authors independently
 proved that in the irrotational case, \eqref{mom}-\eqref{vaccuum}
is globally well-posed for sufficiently small and well-localized initial data.

In the case that $\omega \not = 0$,
Lindblad-Christodoulou \cite{Christodoulou2000}
 used the Taylor sign condition \eqref{tsc} to prove energy estimates
for the system \eqref{mom}-\eqref{vaccuum} in the case that $\D_t$ is
a bounded domain,
and later Lindblad \cite{Lindblad2003} proved that this problem is locally well-posed in Sobolev
spaces using a Nash-Moser iteration. The same result was later shown
by Coutand-Shkoller \cite{Shkoller2007} using a tangential smoothing
operator as well as by \cite{Shatah2010}
who used a more geometric approach which also applies on an unbounded domain.

Relatively little is known about the long-term behavior of solutions
to the problem \eqref{mom}-\eqref{vaccuum} with nonzero vorticity.
We recall that
in the case without free boundary
and without gravity:
\begin{alignat}{2}
 \pa_t + v^k\pa_k v_i + \pa_i p &=0 &&\quad \textrm{ in } \R^3,
 \label{momr3}\\
 \div v &= 0 &&\quad \textrm{ in } \R^3,
 \label{massr3}
\end{alignat}
non-trivial vorticity is the obstacle to obtaining a
global-in-time solution. By
 \cite{Beale1984}, if there are constants $M_0, T_*$ so that
 if $T < T_*$ and $v \in C([0,T] ; H^s(\R^3)) \cap C^1([0,T]; H^{s-1}(\R^3))$
 solves \eqref{momr3}-\eqref{massr3} and the a priori estimate:
\begin{align}
 \int_0^T ||\omega(s)||_{L^\infty(\R^3)}\, ds \leq M_0,
 \label{}
\end{align}
holds,
then the solution can be extended to $v \in C([0, T^*]; H^s(\R^3))
\cap C^1([0,T^*]; H^s(\R^3))$.
It then follows from the fact that:
\begin{align}
 (\pa_t + v^k\pa_k) \omega = \omega \cdot \pa v
 \label{vorteq}
\end{align}
and this result that if $\omega = 0$ at $t = 0$, sufficiently regular
solutions to \eqref{momr3}-\eqref{massr3} can be extended to $T = \infty$.
See also \cite{Ferrari1993} for
an extension to the case of a fixed domain with Neumann boundary condition
and \cite{Ginsberg2018} for an extension to the free-boundary problem on a bounded domain.

In
\cite{Ionescu2018a}, the authors consider the Euler-Maxwell one-fluid system
with nontrivial vorticity but without free boundary
in three dimensions.
They prove that there is
a norm $||\cdot||$ so that if $||\curl v(0, \cdot)|| \leq \delta$ for sufficiently
small $\delta$, then one can continue the solution up to
$T \sim \delta^{-1}$. In particular, this provides a proof of global existence
when $\curl v(0,\cdot) = 0$ for the Euler-Maxwell system.

Returning to the free boundary problem, to the best of our knowledge,
the only papers that address the issue
of the long-time behavior of solutions in the prescence of nontrivial vorticity
are \cite{Ifrim2015},\cite{Bieri2015} and
 \cite{Wang2015}. In \cite{Ifrim2015} Ifrim-Tataru prove that in two space
 dimensions (with one-dimensional boundary),
 solutions
with constant vorticity can be continued up to $T \sim \ve^{-2}$ if the
initial data is of size $\ve$. This is in constrast to the lifespan $T \sim \ve^{-1}$
which is guaranteed by the local well-posedness theory. See also
\cite{Bieri2015} in which Bieri-Miao-Shahshahani-Wu  prove a similar
result for a self-gravitating liquid occupying a bounded region.
In \cite{Wang2015},
the authors
consider the problem in arbitrary dimension and prove that the solution
can be continued so long as the mean curvature of the boundary and
$||\nabla v||_{L^\infty(\D_t)}$ are bounded.

For our result, we
we will measure the regularity of $\omega$ in the norm:
\begin{equation}
 ||\omega(t)||_{H^{r}_w(\D_t)}^2
 = \sum_{k \leq r} \int_{\D_t} (1 + |x|^2 + y^2)^{2} |\pa_{x,y}^k \omega(t,x,y)|^2\,dxdy,
 \label{weighted}
\end{equation}
and we will be considering solutions of Euler's equation
with $\omega \cdot \n|_{\pa \D_t} = 0$.
Our main theorem is an analog of the result in \cite{Ionescu2018a}:
\begin{theorem}
  \label{mainthm}
  Fix $N_1 \geq 6$
  and $N \gg 1$. Define $N_0 = 2N N_1$. There are constants
  $0 < \ve^*_1 \ll \ve_0^* \ll 1$ satisfying the following property.
  Suppose that $v_0, h_0$ satisfy:
  \begin{equation}
    ||v_0||_{H^{N_0}(\D_0)}
   + ||h_0||_{H^{N_0}(\R^2)} \leq \ve_0 \leq \ve_0^*,
   \label{}
  \end{equation}
  and that the Taylor sign condition \eqref{tsc} holds at $t = 0$. Suppose in addition that
   $\omega_0 = \curl v_0$ satisfies the bound:
 \begin{equation}
   ||\omega_0||_{H^{N_1}_w(\D_0)}
  \leq \ve_1 \leq \ve_1^*.
  \label{}
 \end{equation}
Let $(v, h)$ be the solution to \eqref{mom}-\eqref{vaccuum}
with initial data $v_0, h_0$. Let
$T_\omega$ be the largest time so that $(\omega \cdot \n)|_{\pa \D_t} = 0$
for $0 \leq t \leq T_\omega$.
Then the problem \eqref{mom}-\eqref{freebdy} has a unique solution
$(v, h)$ with initial data $(v_0, h_0)$
with $v(t) \in H^{N_0}(\D_t)$, $h(t) \in H^{N_0}(\R^2)$ for
$0 \leq t \leq T_{\ve_0, \ve_1}'$, where:
 \begin{equation}
  T_{\ve_0,\ve_1}' = C_N \min \bigg( \frac{\ve_0}{\ve_1^{1/3}},
  \frac{1}{\ve_0^{N}}, T_\omega \bigg),
  \label{tdef}
 \end{equation}
 for a constant $C_N$ depending only on $N$ and
 $||(-\nabla_\n p_0)^{-1}||_{L^\infty(\pa \D_0)}$.
\end{theorem}
Here, $p_0$ is determined from $v_0, h_0$ by solving:
\begin{alignat}{2}
 \Delta p_0 = -(\pa_i v_0^j)(\pa_j v_0^i),
 && \quad &\text{ in } \D_0,\\
 p_0 = 0, && \quad &\text{ on } \pa \D_0.
 \label{}
\end{alignat}

One simple way to ensure that the condition $(\omega \cdot \n)|_{\pa \D_t} = 0$
holds
for all time is to assume that $\omega_0|_{\pa \D_0} = 0$, since by the transport
equation \eqref{vorteq} it then follows that $\omega|_{\pa \D_t} = 0$ for
$t > 0$ as well
(see Lemma \ref{dtomegabd}).
We therefore have the following corollary:
\begin{cor}
  \label{mainthm2}
 With the same hypotheses as Theorem \ref{mainthm}, suppose in addition that
 $\omega_0|_{\pa \D_0} =0 $. Then the solution $(v, h)$ can be continued
 until:
 \begin{equation}
  T_{\ve_0, \ve_1} = C_N \min \bigg( \frac{1}{\ve_1^{1/3}},
  \frac{1}{\ve_0^N}\bigg).
  \label{}
 \end{equation}
\end{cor}

In particular, if $\omega_0 = 0$, this gives a proof that
the solution can be continued until $T \sim \ve_0^{-N}$. See
also \cite{Berti2017} for a similar lifespan bound for
irrotational water waves on a periodic domain.
Let us make a few remarks.
The assumption that $(\omega\cdot \n)|_{\pa \D_t}$ for $t \geq 0$
is crucial here; as we will see in Section \ref{bdyeqssec}, this allows us to derive
an equation for the evolution of the variables on the boundary which we will
need in order to prove dispersive estimates.

Next, by the results \cite{Germain2012},
\cite{Wu2010}, comparing
to the result in \cite{Ionescu2018a},
one would expect to be able to take $T_{\ve_0,\ve_1}
\sim \frac{\ve_0}{\ve_1}$ which would in turn give a new proof
of global existence in the irrotational case.
The difference between that work and this one is
that solutions to the linearization of the system
\eqref{mom}-\eqref{mass} with zero vorticity decay
at a rate $1/t$, while in \cite{Ionescu2018a}, solutions to the
linearized system decay
at a rate $1/t^{1+\beta}$ for small $\beta$.

We also remark that at the heuristic level the vorticity satisfies an equation of the
form $W' = W^2$ which has lifespan $\sim 1/W_0$ and not $\sim 1/W_0^{1/3}$.
We hope to address both of these issues in future work.

\subsection{Outline of the proof}

As in other works on the global behavior of solutions to
dispersive equations, the result follows from a bootstrap
argument, consisting of energy estimates to control the $L^2$-based norms
and dispersive estimates to control the $L^\infty$-based norms.
We start with the energy estimates.

The system \eqref{mom}-\eqref{freebdy} has the following conserved quantity:
\begin{equation}
 E_0(t) = \int_{\D_t} |v(t)|^2  dxdy+ \int_{\R^2} |h(t)|^2\, dS.
 \label{conservedint}
\end{equation}
Here, we are writing $\D_t = \{(x, y) | x \in \R^2, y \leq h(t,x))\}$.
In the case $\omega = 0$, one can use that the system
\eqref{mom}-\eqref{freebdy} reduces
to a Hamiltonian system on the boundary
(see \eqref{hamil1}-\eqref{hamil2})
and this leads to higher-order energy estimates.
Since we are considering the case $\omega \not=0$, we prove energy estimates
for the system \eqref{mom}-\eqref{freebdy} directly. These energy
estimates are based on the estimates in
\cite{Christodoulou2000}, and we extend their approach to the case
of an unbounded domain.
(See also \cite{Luo2017} where similar
estimates were proved for the compressible Euler equations
with
free boundary in an unbounded domain)

 The energies are of the form:
\begin{align}
 \E^r(t) = \int_{\D_t} Q(\na^r v, \na^r v)\, dxdy + \int_{\pa \D_t}
 |\tn^{r-2}\theta|^2 (-\nabla_\n p)^{-1}\, dS
 + \int_{\D_t} |\na^{r-1} \omega|^2\, dxdy
 \label{intenergy}
\end{align}
where $\na$ is the covariant derivative in $\D_t$, $\tn$ is the covariant
derivative on $\pa \D_t$ and
$\theta$ is the second fundamental form of $\pa \D_t$; writing $\n$ for the
unit normal to $\pa \D_t$ and $\Pi_i^j = \delta_i^j - \n_i \n^j$ for the projection
to the tangent space at the boundary, it is given by:
\begin{align}
 \theta_{ij} = \Pi_{i}^k\Pi_{j}^\ell \na_k \n_\ell.
 \label{}
\end{align}
Here $Q$ is a quadratic form which is the usual norm
$Q(\beta, \beta) = |\beta|^2$ away from the boundary and which
is the norm of the projection to the tangent space at the boundary
when restricted to the boundary, $Q(\beta, \beta) = |\Pi \beta|^2$.
See Section \ref{energyestimates} for a precise
definition. These energies appear to lose control over normal derivatives
of $v$ near the boundary, but it follows from the elliptic estimates in section
\ref{ellsec} (see, in
particular
Lemma \ref{coerlem}) that $\E^r$ controls $||v||_{H^r(\D_t)}^2$.
In Theorem \ref{enestthm}, we prove that:
\begin{equation}
 \frac{d}{dt} \E^r(t) \lesssim \A(t) \Big(\E^r(t) + \A(t) P(\E^{r-1}(t),...,
 \E^0(t))\Big),
 \label{}
\end{equation}
where $P$ is a homogeneous polynomial with positive coefficients
and $\A$ is given by:
\begin{equation}
 \A(t) = ||\na v(t)||_{L^{\infty}(\D_t)} + ||\theta(t)||_{W^{2,\infty}(\pa \D_t)}
 + ||\na p(t)||_{L^\infty(\D_t)} + ||\na^2 p(t)||_{L^\infty(\pa\D_t)}
 + ||\na D_t p||_{L^\infty(\pa \D_t)}.
 \label{}
\end{equation}

We now turn to the more difficult task of proving dispersive estimates,
and for this we will need to change variables. In $\D_t$, we write:
\begin{equation}
 v = \nabla_{x,y} \psi + v_\omega, \quad \Delta_{x,y} \psi = 0,
 \label{}
\end{equation}
with $\nabla_\n \psi = v\cdot \n$ on $\pa \D_t$ and
where $\curl v_\omega = \omega$. We also write $\varphi = \psi|_{\pa \D_t}$.
It will be important that
the energies $\E^r$ control norms of $\varphi, h$ and $v_\omega$. To see why
this is the case, note
that $\theta_{ij} = (1 + |\nabla h|^2)^{-1/2} \nabla_i \nabla_j h$ for $i,j = 1,2$
and that
$||h||_{L^2(\R^2)}^2$ is bounded by the conserved energy, from which it
follows that $||h||_{H^{r}(\R^2)}^2 \lesssim \E^r$. To control $\varphi$,
we start with the observation that:
\begin{equation}
 \int_{\pa \D_t} \varphi \N \varphi\, dS = ||\nabla_{x,y} \psi||_{L^2(\D_t)}^2
 \leq ||v||_{L^2(\D_t)}^2,
 \label{}
\end{equation}
where $\N$ is the Dirichlet-to-Neumann map. The left hand side controls
$||\Lambda^{1/2} \varphi||_{L^2(\pa \D_t)}$ where $\Lambda = |\nabla|$.
To control higher derivatives, we could repeat this argument with $\varphi$
replaced by $\nabla^{r}\varphi$ but this would require controlling the
commutator $[\N, \nabla^r]$ which is nontrivial. Instead it will
suffice for our purposes to use a slightly weaker version of the
trace inequality \eqref{trace} and estimate:
\begin{equation}
 ||\nabla_x \psi||_{H^{r-1}(\pa \D_t)}^2 \lesssim ||\nabla_{x,y} \psi||_{H^{r}(\D_t)}^2
 \lesssim ||v||_{H^{r}(\D_t)}^2 + ||v_\omega||_{H^r(\D_t)}^2
 \lesssim \E^r,
 \label{}
\end{equation}
where the estimate for $||v_\omega||_{H^r(\D_t)}$ follows from the
elliptic estimates in Section \ref{ellsec}, since
$\curl v_\omega = \omega$ and $v_\omega \cdot \n|_{\pa \D_t} = 0$. See
Proposition \ref{uenests}.  To highest order, the left-hand side
here controls $||\nabla_x \varphi||_{H^{r-1}(\R^2)}$.

With the $L^2$ estimates out of the way, we now want to prove $L^\infty$
estimates for $\varphi, h$. In section
\ref{bdyeqssec}, we derive a system satisfied by $\varphi$ and $h$.
 This system is well-known in the case that
$\omega = 0$ (see e.g. \cite{Craig1993}) but the formulation that we use appears
to be new in the case $\omega \not = 0$. To motivate this formulation,
we recall the basic idea behind the ``good unknown''
introduced in \cite{Alazard2014}. We write $V^i = v^i\big|_{\pa \D_t}$, $i = 1,2$ and
$B = v^3\big|_{\pa \D_t}$ as well as $U = V + \nabla h B$.
\footnote{The good unknown used in \cite{Alazard2014} is actually given by $U = V + T_{\nabla h} B$
where $T$ is Bony's paraproduct but it will suffice to use this simpler
definition for our purposes.}
After restricting Euler's equation
\eqref{mom} to $\pa \D_t$ and using the boundary condition
\eqref{vaccuum},
$V$ and $B$ satisfy the following equations:
\begin{align}
 \Dtb V &= -a \nabla h,\\
 \Dtb B &= a - 1,
 \label{}
\end{align}
where $a = (\pa_y p)|_{\pa \D_t}$ and $\Dtb = \pa_t + V^1 \pa_1  + V^2 \pa_2$.
In particular, we have:
\begin{align}
 \Dtb (V + \nabla h B) = - \nabla h - \Dtb \nabla h.
 \label{simple}
\end{align}

In the case $\omega = 0$, $V = \pa_x \psi\big|_{\pa \D_t}$
and $B = \pa_y \psi\big|_{\pa \D_t}$, so by the chain rule, we have:
\begin{align}
 \nabla_x \varphi (x) =
 (\nabla_x \psi)(x, h(x)) + \nabla_x h(x) (\nabla_y \psi)(x, h(x))
 = V + \nabla h B,
 \label{}
\end{align}
with $\varphi(x) = \psi(x, h(x))$. Plugging this into \eqref{simple} gives
an evolution equation for $\nabla \varphi$. It turns out that in the
irrotational case, after making this substitution
\eqref{simple}, is of the form:
\begin{align}
 \pa_t \nabla \varphi = \nabla F( \varphi, h),
 \label{}
\end{align}
for a nonlinearity $F(\varphi, h)$ which also depends on the derivatives
of $\varphi, h$. This leads to an
equation for $\pa_t \varphi$.

When $\omega \not = 0$, we write $v = \nabla_{x,y} \psi + v_\omega$ in
$\D_t$ and let $V_\omega^i = v_\omega^i|_{\pa \D_t}$ for $i = 1,2$
and $B_\omega = v_\omega^3|_{\pa\D_t}$. Repeating the above calculation
leads to an equation of the form:
\begin{align}
 \pa_t \nabla \varphi + \pa_t (V_\omega + \nabla h B_\omega) = \nabla F(\varphi, h)
 + G(\varphi, h, V_\omega, B_\omega),
 \label{lesssimple}
\end{align}
with the same $F$ as above.
Writing $U_\omega = V_\omega + \nabla h B_\omega$, the crucial observation is that:
\begin{align}
 \curl_2 U_\omega = \pa_1 U_\omega^2 - \pa_1 U_\omega^2 =
 \omega \cdot \n && \textrm{ on } \R^2.
 \label{}
\end{align}
See Theorem \ref{formulationthm}.
In particular, if $\omega \cdot \n = 0$ it follows that $U_\omega = \nabla
a_\omega$ for a function $a_\omega$. Making this substitution in
\eqref{lesssimple}, it turns out that $G$
is a gradient, $G
= \nabla H(\varphi, h, V_\omega, B_\omega)$ for some other nonlinearity $H$,
and the system becomes:
\begin{align}
 \pa_t (\nabla \varphi + \nabla a_\omega) =
 \nabla \big( F(\varphi, h) + H(\varphi, h, V_\omega, B_\omega)\big),
 \label{naphieq}
\end{align}
which gives an evolution
equation for $\varphi_\omega = \varphi + a_\omega$. Setting $u =
h + i \Lambda^{1/2} \varphi_\omega$ and writing $w = (V_\omega, B_\omega)$
\eqref{naphieq} and \eqref{freebdy2} lead to an equation of the form:
\begin{align}
 (\pa_t  + i \Lambda) u = N(u)
 + L(w) + N_1(u,w) + N_2(w,w),
 \label{ueq}
\end{align}
where $N, N_1, N_2$ are a nonlinear operators and $L$ is linear.
See Proposition \ref{formulation} for the precise form of the right-hand side.

The nonlinearity $N$ is the same one
that occurs in \cite{Germain2012}, and can be handed using simple modifications of the
arguments there. Specifically, we start with the
Duhamel representation of system \eqref{ueq}:
\begin{align}
 e^{it\Lambda^{1/2}} u(t) &= u_0 + \int_0^t e^{is\Lambda^{1/2}}
 N(u)\, ds
 + \int_0^t e^{is\Lambda^{1/2}} \bigg( L(w) + N_1(u, w) + N_2(w,w)\bigg)\, ds
 \\
 &\equiv u_0 + f_1(u) + f_2(u, w).
 \label{duhamel}
\end{align}
We follow \cite{Germain2012} and define:
\begin{equation}
 ||u||_X = (1+t) ||u||_{W^{4,\infty}(\R^2)}
 + (1+t)^{-\delta} \big( ||u||_{H^{N_0}(\R^2)}
 + ||\Lambda^{\iota} x e^{it\Lambda^{1/2}} u||_{L^2(\R^2)}\big),
 \label{}
\end{equation}
where here $\iota$, $\delta$ are sufficiently small constants.

Minor changes to the arguments in \cite{Germain2012}
(which we outline in Section \ref{realdispsec}) show that:
\begin{align}
 (1+t)|| e^{-it\Lambda^{1/2}} f_1(u)||_{W^{4,\infty}(\R^2)}
 &\lesssim ||u(t)||_{X}^2
 + (1 + t)^{2+\delta} ||\omega(t)||_{H^{N_1}_w(\D_t)},\\
 (1+t)^{-\delta}||\Lambda^{\iota} xf_1(u)||_{L^2(\R^2)}
 &\lesssim
 ||u(t)||_{X}^2
 + (1 +t)||\omega(t)||_{H^{N_1}_w(\D_t)}.
 \label{}
\end{align}

Next,
in Section \ref{vortests}, we establish bounds of the above form for $f_2$:
\begin{align}
 (1+t)|| e^{-it\Lambda^{1/2}} f_2(u,w)||_{W^{4,\infty}(\R^2)}
 &\lesssim ||u(t)||_{X}^2
 + (1 + t)^{2+\delta} ||\omega(t)||_{H^{N_1}_w(\D_t)},\label{uw1}\\
 (1+t)^{-\delta}||\Lambda^{\iota} xf_2(u,w)||_{L^2(\R^2)}
 &\lesssim
 ||u(t)||_{X}^2
 + (1 +t)||\omega(t)||_{H^{N_1}_w(\D_t)}.
 \label{uw2}
\end{align}
The proof of \eqref{uw1}-\eqref{uw2} requires bounding norms
of $w = (V_\omega, B_\omega)$ on the boundary in terms of $\omega$ in the
 interior, for which we use the elliptic estimates in Section
 \ref{ellsec}.
 These estimates combined with the above
the above energy estimates and a continuity
argument show that the solution can be continued until $T \sim
T_{\ve_0,\ve_1}$.

\section{Proof of the main theorem}
\label{pfsec}

We begin by decomposing our initial velocity $v_0$ into its irrotational
and rotational parts.
Given $h_0 : \R^2 \to \R$, set
$\D_0 = \{(x_1, x_2, y) | y \leq h_0(x_1, x_2)\}$. We now write
 $v_0 = \nabla_{x,y} \psi_0 + v_{\omega_0}$, where
$\Delta\psi_0 = 0$ in $\D_0$, $\nabla_\n \psi_0 = v_0 \cdot \n$ on $\pa\D_0$,
and where $\curl v_{\omega_0} = \omega_0 \equiv \curl v_0$. We also write $V_{\omega_0}^i =
v_{\omega_0}^i|_{\pa \D_0},
 i = 1,2$ and $B_{\omega_0} = v_{\omega_0}^3|_{\pa \D_0}$. In Section \ref{bdyeqssec},
  we prove
 that if $\omega_0|_{\pa \D_0} = 0$, then $V_{\omega_0} + \nabla h_0 B_{\omega_0}
 = \nabla a_{\omega_0}$ for a function $a_{\omega_0}$. We then write $\varphi_0 =
 \psi_0|_{\pa \D_0}$ as well as $\varphi_{\omega_0} = \varphi_0 + a_{\omega_0}$ and
 $u_0 = h_0 + i \Lambda^{1/2} \varphi_{\omega_0}$, where $\Lambda = |\nabla|$.

We now fix $N_1 \geq 6$, $N \gg 1$ and set $N_0 = 2NN_1$.
 With the above notation and with $||\cdot||_{H^{N_1}_w}$ defined by \eqref{weighted},
  we suppose that $v_0, \omega_0, h_0$ satisfy:
\begin{align}
  \omega_0\cdot n_0|_{\pa \D_0} &= 0, \\
 ||v_0||_{L^\infty(\D_t)} + ||v_0||_{H^{N_0}(\D_0)} + ||h_0||_{H^{N_0}(\D_0)}
 + ||u_0||_{W^{4,\infty}(\R^2)}
 + ||\Lambda^\iota x u_0||_{L^2(\R^2)} &\leq \frac{1}{2} \ve_0,
 \label{enbds0}\\
 ||\omega_0||_{H^{N_1}_w(\D_0)} &\leq \frac{1}{2} \ve_1 \ll \ve_0,
 \label{need4tsc}\\
\end{align}
for sufficiently small $\ve_0$ and $\iota$, where $n_0$ is the unit
normal to $\pa \D_0$.

We now
define $p_0 : \D_0 \to \R$ by:
\begin{alignat}{2}
 \Delta p_0 &= -(\pa_i v_0^j)(\pa_j v_0^i)
 &&\quad\textrm{ in } \D_0,\\
 p_0 &= 0 &&\quad \textrm{ on } \pa \D_0.
 \label{p0def}
\end{alignat}
In order for the initial value problem
\eqref{mom}-\eqref{initial} to be well-posed, we need
to ensure that $(-\nabla_{\n_0} p_0) \geq \delta_0 > 0$ for some $\delta_0$.
In the irrotational case, this condition holds automatically, essentially
because then $\Delta p_0 = -(\pa v)^2 \leq 0$ (see \cite{Wu1999}). When $\curl v_0 \not =0$
we instead have the following result:
\begin{lemma}
  \label{inttscprop}
  Suppose that $||\omega_0||_{L^\infty(\D_0)} \leq
  \frac{1}{2} ||v_0||_{L^\infty(\D_0)}$.
  Then, with $p_0$ defined
 by \eqref{p0def}, there is a constant $c_0 > 0$ so that:
 \begin{equation}
  (-\nabla_{\n_0} p_0) \geq 2c_0 > 0 \textrm{ on } \pa \D_0.
  \label{}
 \end{equation}
\end{lemma}

\begin{proof}
 We follow the argument in \cite{Wu1999}. We fix a function $f: \pa \D_0 \to \R$
 and let $F$ denote its harmonic extension to $\D_0$. By Green's identity:
 \begin{equation}
  \int_{\pa \D_0} f \nabla_{\n_0} (p_0 + y) - (p_0 + y) \nabla_{\n_0} f
  = \int_{\D_0} \Delta (p_0 + y) F.
  \label{greenident1}
 \end{equation}
 We now note that $\Delta p_0 = -(\pa_i v^j_0)(\pa_j v_0^i)
 = -(\pa_i v^j_0)(\pa_i v_0^j)
 + (\pa_i v^j_0) \delta^{i\ell}(\curl v_0)_{j\ell}$. By assumption we have that
 $||\pa v_0 - \curl \omega||_{L^\infty(\D_t)} \geq \frac{1}{2}
 ||\pa v_0||_{L^\infty(\D_t)}$ and
 so in particular we have that $\Delta p_0
 < 0.$ Therefore by \eqref{greenident1} and the
 fact that $p_0 = 0$ on $\pa \D_0$,  we have:
 \begin{equation}
  \int_{\pa \D_0} f \na_\n (p_0 + y) - y \na_\n f
  > 0.
  \label{}
 \end{equation}
 The rest of the proof of Lemma 4.1 from \cite{Wu1999} now goes through without change.
\end{proof}

We will use the following local well-posedness result, which
follows from Theorem B in \cite{Shatah2010} and the above lemma:
\begin{prop}
  \label{lwp}
  Let
   $h_0 \in H^{N_0}(\R^2)$, $\D_0 = \{ (x_1, x_2, y) | y \leq h_0(x_1, x_2)\}$
  $v_0 \in H^{N_0}(\D_0)$. Suppose that
  $||\omega_0||_{L^\infty(\D_0)} \leq \frac{1}{2} ||v_0||_{L^\infty(\D_0)}$.
  Then there is a $T = T(v_0, h_0) > 0$, a function $h:[0,T] \times \R^2 \to \R$
   and
  a vector field $v = v(t)$ defined on  $\D_t \equiv \{(x,y) | y \leq h(t,x)\}$
  for $0 \leq t \leq T$,
  so that $v\big|_{t = 0} = v_0,
  \D_t \big|_{t = 0}  = \D_0$, $(v,\D_t)$ satisfy
  \eqref{mom}-\eqref{vaccuum} and $v(t,\cdot) \in H^{N_0}(\D_t)$,
  $h(t,\cdot) \in H^{N_0}(\R^2)$ for $t \leq T$.
\end{prop}

We now want to extend the time $T$ in this theorem to $T'_{\ve_0, \ve_1}$
defined in \eqref{tdef}, provided that the vorticity $\omega$ vanishes on
$\pa\D_t$.
We suppose
that $u, v$ satisfy the following bootstrap assumptions
for $t \geq 0$:
\begin{align}
 ||u(t)||_{W^{4,\infty}(\R^2)} &\leq \frac{\ve_0}{1 + t}
\label{bootstrap1}\\
 ||v(t)||_{H^{N_0}(\D_t)} +
 ||h(t)||_{H^{N_0}(\R^2)}
 + ||\Lambda^\iota xe^{it\Lambda^{1/2}} u(t)||_{L^2(\R^2)}
  &\leq \ve_0(1 + t)^\delta,\label{bootstrap2}\\
 ||\omega(t)||_{H^{N_1}_w(\D_t)}
 &\leq \ve_1 (1+t)^\delta,
 \label{bootstrap3}
\end{align}
and that $\omega\cdot \n|_{\pa \D_t} = 0$,
where here $\iota > 0$ is a small constant.
In Section \ref{coerensec} we show that:
\begin{equation}
 ||\Lambda^{1/2} \varphi_\omega||_{H^{N_0-1}(\R^2)} \lesssim
 ||v(t)||_{H^{N_0}(\D_t)} + ||h(t)||_{H^{N_0}(\pa\D_t)}
 + O(\ve_0^2),
 \label{}
\end{equation}
if \eqref{bootstrap1}-\eqref{bootstrap2} hold.
In particular, the assumption \eqref{bootstrap2} implies an estimate for
$||u||_{H^{N_0-1}(\R^2)}$,
 a fact
which is used several times in the proofs of the following theorems.

Recalling the definitions in \eqref{duhamel}, and that we are writing
$w = (V_\omega, B_\omega)$, we have:
\begin{prop}
  \label{dispprop1}
 If the bootstrap assumptions \eqref{bootstrap1}-\eqref{bootstrap3} hold
 for sufficiently small $\ve_0,\ve_1$ and $\omega\cdot \n|_{\pa \D_t} = 0$ for
 $0 \leq t \leq T$, then:
 \begin{align}
  ||e^{-it\Lambda^{1/2}} f_1(u)||_{W^{4,\infty}(\R^2)}
  + ||e^{-it\Lambda^{1/2}} f_2(u,w)||_{W^{4,\infty}(\R^2)}
  \lesssim \frac{\ve_0^{2}}{1 + t} + \ve_1 (1+t)^{1+2\delta},\label{fbdsA}\\
  ||\Lambda^\iota (x f_1(u))||_{L^2(\R^2)} +
  ||\Lambda^\iota (x f_2(u,w))||_{L^2(\R^2)} \lesssim \ve_0^{2} (1 + t)^{\delta} +
  \ve_1(1+t)^{2+2\delta},
  \label{fbdsB}
 \end{align}
 for $0 \leq t \leq T$.
\end{prop}
The term $f_1$ can be estimated by simple modifications of the estimates
in and we outline this approach in Section \ref{realdispsec}.
The estimates for $f_2$ can be found in Sections \ref{vortest1}-\ref{g4ests}.

We now need some estimates to control the size of $\omega$, which
we prove in Section \ref{vortensec}:
\begin{prop}
  \label{intvortprop}
  If the assumptions \eqref{bootstrap1}-\eqref{bootstrap3} hold, then
  there is a constant $C_N$ so that:
 \begin{align}
   ||\omega(t)||_{H^{N_1}_w(\D_t)}^2
   \leq ||\omega(0)||_{H^{N_1}_w(\D_0)}^2
   + C_N\big( \ve_0 (1 + t)^{1/N}\big) \big( \ve_1 (1+t)^{1+\delta}\big)
   \ve_1^2 (1 + t)^{2\delta}.
  \label{}
 \end{align}
\end{prop}
In particular, if:
\begin{align}
 2C_N t \leq T_1 \equiv \min \bigg( \frac{1}{ \ve_1^{1/3}},
 \frac{1}{\ve_0^{N}}\bigg),
 \label{t1def}
\end{align}
this implies that:
\begin{equation}
 ||\omega(t)||_{H^N_w(\D_t)}^2 \leq \frac{3}{4} \ve_1^2 (1 +t)^{2\delta}.
 \label{}
\end{equation}

The last ingredient we need is an energy estimate for the
entire system, which we prove in Section \ref{energyestimates}:
\begin{prop}
  \label{intenprop}
If $||v_0||_{H^{N_0}(\D_0)}
 + ||\omega_0||_{H^{N_0-1}(\D_0)} \leq \ve_0/2$ and
 the bootstrap assumptions \eqref{bootstrap1}-\eqref{bootstrap3}
 hold, then with $c_0$ as in Lemma \ref{inttscprop}, there is
 a constant $C_{N_0}^E = C_{N_0}^E(c_0)$ so that:
 \begin{equation}
   ||v(t)||_{H^{N_0}(\D_t)}^2
   + ||h(t)||_{H^{N_0}(\R^2)}^2
   \leq \frac{\ve_0^2}{4} + C_{N_0}^E \big( \ve_0
   + \ve_1 (1+t)^{2 +\delta}\big) \ve_0^2 ( 1 + t)^{2\delta}.
  \label{mainenergy}
 \end{equation}
\end{prop}
In particular, if $t$ is such that:
\begin{equation}
 2C_{N_0}^E t\leq T_2 \equiv \frac{1}{\ve_1^{1/3}},
 \label{t2def}
\end{equation}
this implies that:
\begin{equation}
 ||v(t)||_{H^{N_0}(\D_t)}^2 +
 ||h(t)||_{H^{N_0}(\D_t)}^2
 \leq \frac{\ve_0^2}{4}
 + \ve_0^{3}(1 + t)^{2\delta}.
 \label{}
\end{equation}

Setting $T_{\ve_0, \ve_1} = \min(T_1, T_2)$, a standard continuity
argument then gives Theorem \ref{mainthm}.

\section{Derivation of the equations on the boundary}
\label{bdyeqssec}
We will use the equations \eqref{mom}-\eqref{vaccuum} directly to
prove energy estimates.
However, to prove the dispersive estimates in Proposition
\ref{dispprop1}, we will need to use equations
for $h$ and $v\big|_{\pa \D_t}$.
In the irrotational case, $v_i = \pa_i \psi$ for a harmonic function
$\psi$ satisfying $\nabla_\n \psi = v\cdot \n$ on $\pa \D_t$. Letting $\varphi =
\psi|_{\pa \D_t}$, one can show that $h, \varphi$ satisfy the system:
\begin{align}
 \pa_t h &= G(h) \varphi,\label{hamil1}\\
 \pa_t \varphi &= -h - \frac{1}{2} |\nabla \varphi|^2 + \frac{(G(h) \varphi +\nabla h\cdot
 \nabla \varphi)^2}{2(1+ |\nabla h|^2)},
 \label{hamil2}
\end{align}
where $G(h)$ is the rescaled Dirichlet-to-Neumann map (see \eqref{dndef})
and we are writing $\nabla = (\pa_1, \pa_2)$.
 This system is derived
from the fact that
when $\omega = 0$, Euler's equations become:
\begin{align}
 \pa_i \big( \pa_t \psi + |\pa \psi|^2 + p + y\big) = 0, &&
 \textrm{ in } \D_t.
 \label{}
\end{align}
See e.g. \cite{Craig1993} or \cite{Lannes2013} for a derivation.

This no longer works when $\omega \not = 0$ and so another approach
is needed.
Our derivation of the equations on the boundary is partially based on the ideas
in \cite{Alazard2014} (see in particular Section4.1 there).
We define:
\begin{align}
 V^i = v^i|_{\pa \D_t} \textrm{ for } i = 1,2, && B = v^3|_{\pa \D_t}.
\end{align}
In what follows we will write $\pa_i = \frac{\pa}{\pa x^i}$ for $i  = 0,1,2,3$
with the convention that $x^0 = t$. We will also occasionally write $\pa_y = \pa_3$.
We will also write $\nabla$ for the derivative of quantites
defined on $\R^2 \sim \pa \D_t$
and $\pa$ when differentiating quantities defined on $\D_t$.
We now collect a few well-known and elementary identities.
Given $f:\D_t \to \R$, write $F(x) = f(x,h(x)) = f|_{\pa \D_t}(x)$. Then,
by the chain rule:
\begin{align}
 \pa_i F = (\pa_i f)|_{\pa \D_t} + \nabla_i h (\pa_y f)|_{\pa \D_t},
 \quad i = 0,1,2
 \label{cr}
\end{align}
If $f$ is harmonic on $\D_t$ then additionally:
\begin{align}
 (\pa_y f)|_{\pa \D_t} = \frac{1}{1 + |\nabla h|^2}
 \Big( G(h) F + \nabla h \cdot \nabla F\Big),
 \label{neum}
\end{align}
where $G(h)$ is the rescaled Dirichlet-to-Neumann operator:
\begin{align}
 G(h) F = \sqrt{1 + |\nabla h|^2} \nabla_\n f|_{\pa \D_t}.
 \label{dndef}
\end{align}
We also recall that the boundary condition \eqref{freebdy} can
be written:
\begin{equation}
 \pa_t h + V^1\nabla_1 h + V^2 \nabla_2 h = B.
 \label{dth}
\end{equation}
As a consequence, writing $\Dtb = \pa_t + V^1\pa_1
+ V^2\pa_2$, we have:
\begin{align}
 \Dtb F = (D_t f)|_{\pa \D_t}.
 \label{dtb}
\end{align}

Next, in $\D_t$ we define $\psi \in L^6(\D_t)\cap \dot{H}^1(\D_t)$ to be the harmonic extension of
$v\cdot \n$ to $\D_t$, that is, $\psi$ satisfies:
\begin{align}
 \Delta \psi  = 0 \quad \text{ in } \D_t,
 \quad
 \nabla_\n \psi = \n \cdot v \quad \text{ on } \pa \D_t.
 \label{}
\end{align}
The function $\psi$ is unique since $\div v = 0$ in $\D_t$.

It then follows that $v_\omega \equiv v - \pa \psi$ satisfies:
\begin{align}
 \curl v_\omega = \omega, \quad \div v_\omega =0, \text{ on } \D_t,
 \quad \n\cdot v_\omega = 0 \text{ on } \pa \D_t.
 \label{}
\end{align}
We will write:
\begin{align}
 \varphi  = \psi|_{\pa \D_t},
 \quad
 V_\omega^i = v_\omega^i|_{\pa \D_t}, \,i = 1,2,
 \quad
 B_\omega = v_\omega^3|_{\pa \D_t}.\label{vomegadef}
\end{align}

The following
derivation is inspired by the approach of \cite{Alazard2014} and
\cite{Wang2015}.
The main result of this section is the following:
\begin{theorem}
  \label{formulationthm}
  With the above notation:
  \leavevmode
  \begin{enumerate}
   \item Writing $U_\omega = V_\omega + \nabla h B_\omega$,
   we have $\nabla_1 U_\omega^2 - \nabla_2 U_\omega^1 = \omega|_{\pa \D_t} \cdot \n,$
   In particular, $U_\omega = \nabla a_\omega$ for a function
   $a_\omega : \R^2 \to \R$.
   \item The variables $\varphi, h,
   V_\omega$ and  $B_\omega$ satisfy the system:
   \begin{align}
    \pa_t h &= G(h)\varphi,\label{whamil1}\\
    \pa_t (\varphi + a_\omega) &= - h - \frac{1}{2} |\nabla \varphi|^2
    + \frac{1}{2} \frac{( G(h) \varphi + \nabla h \cdot \nabla \varphi)^2}{1 + |\nabla h|^2}
    + R_\omega,
    \label{whamil2}
   \end{align}
  where:
  \begin{equation}
   R_\omega = -\frac{1}{2} |V_\omega|^2
   -\nabla \varphi \cdot V_\omega - \frac{1}{2} \big(V_\omega \cdot
   \nabla h\big)^2
   + (G(h) \varphi) V_\omega \cdot \nabla h
   \label{rwdef}
  \end{equation}
  \end{enumerate}
\end{theorem}

\begin{proof}

  By \eqref{cr}:
  \begin{align}
   \nabla_1 U_\omega^2 - \nabla_2 U_\omega^1 &=
   \nabla_1 V_\omega^2 - \nabla_2 V_\omega^1
   + (\nabla_2 h) \nabla_1 B_\omega - (\nabla_1 h) \nabla_2 B_\omega\\
   &= (\pa_1 v_\omega^2)|_{\pa \D_t} - (\pa_2 v_\omega^1)|_{\pa \D_t}
   + (\nabla_1 h)  (\pa_3 v_\omega^2)|_{\pa \D_t}
   - (\nabla_2 h) (\pa_3 v_\omega^1)|_{\pa \D_t}\\
   &+ (\nabla_2 h) (\pa_1 v_\omega^3)|_{\pa \D_t} -
   (\nabla_1 h)(\pa_2 v_\omega^3)|_{\pa \D_t}
   + (\nabla_2 h) (\nabla_1 h) (\pa_3 v_\omega^3)|_{\pa \D_t}
   - (\nabla_1 h) (\nabla_2 h)(\pa_3 v_\omega^3)|_{\pa \D_t}\\
   &= (\pa_1 v_\omega^2 - \pa_2 v_\omega^1)|_{\pa \D_t}
   -\bigg( (\nabla_1 h) (\pa_2 v_\omega^3 - \pa_3 v_\omega^2)|_{\pa \D_t}
   + (\nabla_2 h) (\pa_3 v_\omega^1 - \pa_1 v_\omega^3)|_{\pa \D_t}\bigg)\\
   &= \omega^3|_{\pa \D_t} - (\nabla_i h) \, \omega^i|_{\pa \D_t},
   \label{cancellation}
  \end{align}
  which gives the first result.

  We now derive \eqref{whamil1} -\eqref{whamil2}.
Differentiating \eqref{dth} and using the fact that $[\Dtb, \nabla]
= -\nabla V^k\nabla_k$ gives:
\begin{align}
 \Dtb \pa_i h &= \nabla_i B -
 \nabla_i V^k\nabla_k h
 \label{dtpah}
\end{align}
Writing $a = (\pa_y p)\big|_{\pa \D_t}$, restricting Euler's equations
\eqref{mom}
to the boundary and using that $p = 0$ on $\pa \D_t$ gives:
\begin{align}
 \Dtb V_i &= -a\nabla_i h, \quad i = 1,2\\
 \Dtb B &= a-1.
 \label{}
\end{align}
Therefore:
\begin{align}
  \Dtb(\nabla \varphi + U_\omega) =
 \Dtb (V + \nabla h B) &= -\nabla h + (\Dtb \nabla h) B\\
 &= -\nabla h + (\nabla B  - \nabla V^k \nabla_k h)B\\
 &= -\nabla h + \frac{1}{2} \nabla |B|^2
 - \nabla V^k\nabla_k h B.
 \label{evol}
\end{align}

We will write $f_b = f\big|_{\pa \D_t}$ for the restriction
to the boundary.
Expanding out the definition of $\Dtb$ and recalling by convention, sums over
repeated upper and lower indices run over only the first two indices gives that:
\begin{align}
 \Dtb (\nabla \varphi + U_\omega)
 &=\pa_t(\nabla \varphi + U_\omega)
 + (\pa_k\psi)_b \nabla^k \nabla \varphi
 + (\pa_k \psi)_b\nabla^k U_\omega
 + V_\omega^k \nabla_k \nabla \varphi
 + V_\omega^k \nabla_k U_\omega
 \label{}
\end{align}
Combining this with \eqref{evol} and expanding $(\pa_k \psi)_b =
\nabla_k \varphi - \nabla^k h(\pa_y \psi)_b$, we have:
\begin{align}
 \pa_t( \nabla \varphi + U_\omega)
 &= -\nabla h + \frac{1}{2} \nabla  |B|^2 - \nabla V^k\nabla_k h B
- \big(\nabla^k\varphi
-(\pa_y \psi)_b \nabla^kh \big)\nabla_k \nabla \varphi\\
&- \nabla^k \varphi \nabla_k U_\omega
+ (\pa_y \psi)_b \nabla^k h\nabla_kU_\omega
-V_\omega^k\nabla_k \nabla \varphi
- V_\omega^k \nabla_k U_\omega.
 \label{}
\end{align}

Expanding $V, B$ in terms of $\psi$ and $v_\omega$ and using \eqref{cr}:
\begin{align}
 \nabla V^k\nabla_k h B
&= \nabla (\pa_k \psi)_b\nabla^k h (\pa_y \psi)_b
+ \nabla V_\omega^k \nabla_k h B_\omega
+ \nabla (\pa_k\psi)_b \nabla^k h B_\omega
+ \nabla V_\omega^k \nabla_k h (\nabla_y\psi)_b\\
&= (\nabla \nabla_k\varphi) \nabla^k  h (\pa_y \psi)_b
+ \nabla V_\omega^k \nabla_k h B_\omega
+ \nabla(\nabla^k\varphi)\nabla_k h B_\omega\\
&- \nabla (\nabla^k h (\pa_y\psi)_b) \nabla_k h B_\omega
- \nabla(\nabla_k h (\pa_y \psi)_b)\nabla^k h(\pa_y\psi)_b
 + \nabla V_\omega^k \nabla_k h (\pa_y\psi)_b
 \label{}
\end{align}

We insert this expression into the previous one to get:
\begin{align}
 \pa_t (\nabla \varphi + U_\omega)
 &= A(\varphi, h)
 + \frac{1}{2} \nabla |B_\omega|^2
 + \nabla ( (\pa_y\psi)_b B_\omega)\\
 &- \nabla V_\omega^k \nabla_k h B_\omega
 - V_\omega^k \nabla_k U_\omega\\
 & -\nabla \nabla^k \varphi (\nabla_k h B_\omega)
 -\nabla^k \varphi \nabla_k U_\omega - V_\omega^k \nabla_k \nabla \varphi\\
 & +(\pa_y \psi)_b \nabla^kh \nabla_k U_\omega + \nabla( \nabla^k h (\pa_y \psi)_b)
 \nabla_k h B_\omega
 - \nabla V_\omega^k \nabla_k h (\pa_y \psi)_b,
 \label{annoying}
\end{align}
where $A$ is given by:
\begin{align}
 A &= -\nabla h + \frac{1}{2} \nabla (\pa_y\psi)_b^2
 -\frac{1}{2} \nabla |\nabla \varphi|^2
 + (\pa_y\psi)_b\nabla^kh \nabla_k \nabla \varphi\\
 &- \bigg(\nabla (\nabla^k \varphi )\nabla_k h
 + \nabla (\nabla^k h (\pa_y\psi)_b) \bigg)(\pa_y\psi)_b\\
 &= -\nabla h + \frac{1}{2} \nabla (\pa_y \psi)_b^2
 - \frac{1}{2} \nabla |\nabla \varphi|^2-\nabla\big(\nabla^k h (\pa_y\psi)_b
 \big) (\pa_y\psi)_b,
 \label{}
\end{align}
using \eqref{cr} in the last step.
Applying \eqref{neum} shows that:
\begin{equation}
 A = \nabla \bigg(-h - \frac{1}{2}|\nabla \varphi|^2
 + \frac{(G(h)\varphi + \nabla h\cdot \nabla \varphi)^2}{2(1 + |\nabla h|^2)}
 \bigg)
 \label{}
\end{equation}
We now want to show that all of the other terms in \eqref{annoying} are
also gradients.

To handle the terms on the second row of \eqref{annoying}, we note that
by the definition of $U_\omega$:
\begin{multline}
  \nabla V_\omega^k \nabla_k h B_\omega + V_\omega^k \nabla_k U_\omega
  = \nabla V_\omega^k (\delta_{k\ell} U_\omega^\ell) + V_\omega^k \nabla_k U_\omega
  - \nabla V_\omega^k (\delta_{k\ell} V_\omega^\ell)\\
  = \nabla (\delta_{k\ell} V_\omega^k U_\omega^\ell) - \frac{1}{2} \nabla |V_\omega|^2
 \label{annoying1}
\end{multline}
where we used the fact that $\curl U_\omega = 0$ in the last step.

To deal with the terms on the third row of \eqref{annoying} we note that:
\begin{align}
 \nabla_k \varphi \nabla^k U_\omega
 + V_\omega^k \nabla_k \nabla \varphi
 +\nabla \nabla_k\varphi \nabla^k h B_\omega
 = \nabla_k \varphi \nabla^k U_\omega
 + \nabla_k \nabla \varphi (V_\omega^k + \nabla^k h B_\omega)
 = \nabla (\nabla_k\varphi U_\omega^k).
 \label{annoying2}
\end{align}

Finally, to handle the terms on the last line of \eqref{annoying}, we
again use that $\curl U_\omega = 0$ and expand out $U_\omega = V_\omega
+ \nabla h B_\omega$ and write the result as:
\begin{align}
 (\pa_y\psi)_b
 &\big(\nabla_k h \nabla_i  U_\omega^k
 +\nabla_i \nabla^kh \nabla_kh B_\omega
 - \nabla_i V_\omega^k \nabla_k h\big)
 + |\nabla h|^2 B_\omega \nabla_i (\pa_y \psi)_b\\
 &= (\pa_y \psi)_b
 \big(\nabla_k h \nabla_i V_\omega^k
 + \nabla_k h \nabla_i (\nabla_kh B_\omega) +
 \nabla_i \nabla_k h \nabla^k  h B_\omega - \nabla_i V_\omega^k \nabla_k h
 \big)
 + |\nabla h|^2 B_\omega \nabla_i(\pa_y \psi)_b\\
 &=\nabla_i \big( |\nabla h|^2 (\pa_y \psi)_b B_\omega\big).
 \label{annoying3}
\end{align}

Combining the results of \eqref{annoying1}-\eqref{annoying3}, we see
that \eqref{annoying} becomes:
\begin{multline}
 \pa_t (\nabla \varphi + \nabla a_\omega)
 = A(\varphi, h)\\
 + \frac{1}{2} \nabla |B_\omega|^2 + \frac{1}{2} \nabla |V_\omega|^2
 - \nabla \big(V_\omega \cdot U_\omega)
 + \nabla ((1 + |\nabla h|^2)(\pa_y\psi)_b B_\omega)
 -\nabla (\nabla\varphi \cdot U_\omega)
 \label{lessannoying}
\end{multline}

Now we note that since $v_\omega \cdot \n = 0$, we have $B_\omega =
V_\omega^k\nabla_k h$ which further implies $U_\omega = V_\omega +
\nabla h  (V_\omega \cdot \nabla h)$. The second line of
\eqref{lessannoying} then becomes the gradient of:
\begin{multline}
  -\frac{1}{2} (\nabla h\cdot V_\omega)^2
  - \frac{1}{2} (V_\omega)^2
  + \bigg( (1 + |\nabla h|^2) (\pa_y\psi)_b - \nabla\varphi\cdot
  \nabla h\bigg) (\nabla h \cdot V_\omega)
  - \nabla \varphi \cdot V_\omega
  \\
  =
  -\frac{1}{2} (\nabla h\cdot V_\omega)^2
  - \frac{1}{2} (V_\omega)^2 +
  (G(h)\varphi) (V_\omega \cdot \nabla h)
  -\nabla \varphi \cdot V_\omega
 \label{}
\end{multline}

Next, we note that the vorticity does not enter into $h$ \eqref{dth} when
we write $v$ in terms of $\psi, v_\omega$.
Indeed, recalling that $v_\omega \cdot \n = 0$ on $\pa \D_t$
and using \eqref{dth} gives:
\begin{align}
 \pa_t h &= -(\pa \psi)_b\cdot \nabla h - V_\omega\cdot \nabla h
 + (\pa_y \psi)_b + B_\omega\\
 &= \sqrt{ 1 + |\nabla h|^2} \big((\pa \psi)_b\cdot \n + (v_\omega)_b \cdot \n
 \big)\\
 &= G(h) \varphi,
 \label{realdth}
\end{align}
where in the last step we used that $\Delta \psi = 0$ in $\D_t$.

Combining the result of the above calculation with \eqref{realdth}
completes the proof.
\end{proof}

It is a little awkward to work in terms of $a_\omega$, since it depends on the
vorticity in the interior in a complicated way,
and moreover we only control
$\nabla a_\omega$, not $a_\omega$ itself. For this reason we set:
\begin{equation}
 \varphi_\omega = \varphi + a_\omega.
 \label{varphiwdef}
\end{equation}
The above system becomes:
\begin{align}
 \pa_t h &= G(h) \varphi_\omega - G(h) a_\omega,\\
 \pa_t \varphi_\omega &= -h - |\nabla \varphi_\omega|^2
 + \frac{(G(h) \varphi_\omega + \nabla h\cdot \nabla \varphi_\omega)^2}{1 + |\nabla h|^2}
 + \widetilde{R_\omega},
 \label{}
\end{align}
where:
\begin{multline}
 \widetilde{R_\omega} =
 -\frac{1}{2} |\nabla a_\omega|^2 + \nabla \varphi_\omega \cdot \nabla a_\omega
 + (1 + |\nabla h|^2)^{-1} \Big( (G(h)a_\omega + \nabla h\cdot \nabla a_\omega)^2
 - 2 (G(h)a_\omega + \nabla h\cdot \nabla a_\omega)(G(h)\varphi_\omega
 + \nabla h\cdot \nabla \varphi_\omega)\Big)\\
 -\frac{1}{2} |V_\omega|^2
   -\nabla \varphi_\omega \cdot V_\omega - \frac{1}{2} \big(V_\omega \cdot
   \nabla h\big)^2
   + (G(h) \varphi_\omega) V_\omega \cdot \nabla h
   +\nabla a_\omega \cdot V_\omega
   - (G(h)a_\omega )V_\omega \cdot \nabla h)
 \label{}
\end{multline}
We now recall that $\nabla a_\omega = V_\omega + \nabla h B_\omega$ and
that $B_\omega = \nabla h \cdot V_\omega$. Writing $G(h)a_\omega
= G(h)\Lambda^{-1} R\cdot \nabla a_\omega$,
we note that $V_\omega$ enters \emph{linearly} into these
equations, since:
\begin{equation}
 \pa_t h = G(h)\varphi_\omega - G(h)(\Lambda^{-1} R\cdot V_\omega)
 - G(h) (\Lambda^{-1} R\cdot (\nabla h  \cdot V_\omega) ).
 \label{}
\end{equation}
We also note that $V_\omega, B_\omega$ enter no more than quadratically into
the remaining terms.

Using these identities, can further re-write:
\begin{multline}
 \widetilde{R_\omega} = -|V_\omega \cdot \nabla h|^2 +
  (\nabla \varphi_\omega\cdot \nabla h)(V_\omega \cdot \nabla h)
  +G(h)\Big[\varphi_\omega - \Lambda^{-1} R\cdot V_\omega - \Lambda^{-1} R\cdot (\nabla hB_\omega) \Big]
  (V_\omega \cdot \nabla h)\\
  +(1 + |\nabla h|^2)^{-1} \Bigg(\Big( G(h) [\Lambda^{-1} R\cdot V_\omega] + \nabla h\cdot V_\omega\Big)^2
  - (G(h)\Lambda^{-1} R\cdot V_\omega + \nabla h\cdot V_\omega)(G(h)\varphi_\omega
  + \nabla h \cdot \nabla \varphi_\omega)\Bigg)\\
  + \text{ more nonlinear terms }
 \label{rearrange}
\end{multline}

We now expand out $G(h)$ in powers of $h$.
 We recall the following
expansion of $G(h)$ in powers of $h$:
\begin{align}
 G(h) = \Lambda + G_2(h) + G_3(h) + G_4(h),
 \label{formal1}
\end{align}
with:
\begin{align}
 G_2(h) &= -\nabla\cdot(h \nabla) + \Lambda(h \Lambda),\\
 G_3(h) &= \Lambda(h^2\Lambda^2) + \Lambda^2(h^2 \Lambda)
 - 2( h \Lambda (h \Lambda)),
 \label{formal2}
\end{align}
and where $G_4(h) \equiv G(h) - \Lambda - G_1(h) - G_2(h)$ vanishes to order
3 when $h = 0$. See \cite{Craig1994} for a formal derivation
of this expansion, and e.g. Appendix F of \cite{Germain2012} for
rigorous estimates for $G_4$. Here, we are using the notation:
\begin{align}
  \Lambda^s f = \F^{-1} (|\xi|^{s} \F f), \quad s \in \R,
 \label{}
\end{align}
where $\F$ is the Fourier transform on $\R^2$.

In particular, keeping track of just the terms which are
linear or quadratic, the above
equations become:
\begin{align}
 \pa_t h &= \Lambda \varphi_\omega - R\cdot V_\omega - \nabla\cdot (h \nabla \varphi_\omega)
 - \Lambda (h \Lambda \varphi_\omega) +
 \Lambda(h R\cdot V_\omega) + \nabla \cdot (h V_\omega) + ...\\
 \pa_t \varphi_\omega &= -h - |\nabla \varphi_\omega|^2 +
 (\Lambda \varphi_\omega)^2 + |R\cdot V_\omega|^2
 + (R\cdot V_\omega)\Lambda \varphi_\omega + ...
 \label{}
\end{align}

We now set:
\begin{equation}
 u = h + i \Lambda^{1/2}\varphi_\omega
 \label{realudef}
\end{equation}
With this definition, we can recover $h, \varphi_\omega$
from $u$:
\begin{align}
 h = \Re u,
 &&
 \varphi_\omega = \Lambda^{-1/2} \Im u.
 \label{}
\end{align}

In what follows, we will write
$u_R = Re u$ and $u_I = \Im u$.  We will also write $R_i$ for the
Riesz transform:
\begin{align}
 \F (R_i f)(\xi) = \frac{\xi_i}{|\xi|} (\F f)(\xi), &&i = 1,2.
 \label{}
\end{align}

\begin{prop}
  \label{formulation}
 With the above definitions, we have:
 \begin{align}
  (\pa_t + i \Lambda^{1/2})u = N(u) + L(V_\omega) + N_1(u,V_\omega) + N_2(u, V_\omega)
  + N_3(u, V_\omega),
  \label{}
 \end{align}
 where $N(u) = B(u) + T(u) + R(u)$ and:
 \begin{align}
  B(u) &=
  \Lambda u_R (\Lambda^{1/2} u_I)
  + \nabla \cdot ( u_R
  (\Lambda^{-1/2} \nabla u_I))
   + i\Lambda^{1/2}
   \Big( |\Lambda^{-1/2} \nabla u_I|^2
   + |\Lambda^{1/2} u_I|^2\Big)
   \\
  T(u) &= -\frac{1}{2} \Lambda( u_R^2 \Lambda^{3/2}u_I) + \Lambda^2 (u_R^2
  \Lambda^{1/2} u_I) - 2\Lambda(u_R\Lambda( u_R\Lambda^{1/2} u_I))
  +i\Lambda^{1/2}\Big(\Lambda^{1/2} u_I \big(u_R\Lambda^{3/2}u_I -
  \Lambda(u_R\Lambda^{1/2} u_I) \big)\Big)
  \label{}\\
  L(V_\omega) &= -R \cdot V_\omega,\\
  N_1(u,V_\omega) &= \Lambda^{1/2} (R\cdot V_\omega \Lambda^{1/2} u_I)
  - \nabla\cdot( u_R V_\omega)
  + \Lambda (  u_R R\cdot V_\omega)\\
  N_2(V_\omega, V_\omega) &= \Lambda^{1/2} (R \cdot V_\omega)^2,
  \label{}
 \end{align}
 and where $R(u)$ (resp. $N_3(u,V_\omega)$) vanish to order 4 (resp. 3) when $h = 0$,
 and where $N_3(u,V_\omega)$ is quadratic in $V_\omega$ and its derivatives.
\end{prop}

For later use, we record the
 Duhamel form of these equations:
\begin{align}
 e^{it\Lambda^{1/2}}  u(t) - u_0 &=
 g_1(t) + g_2(t) + g_3(t) + g_4(t) + g_5(t),
 \label{duha2}
\end{align}
where:
\begin{equation}
 g_1(t) = \int_0^t e^{is\Lambda^{1/2}} N(u)\, ds,
 \label{}
\end{equation}
\begin{align}
 g_2(t)&= \int_0^t e^{is\Lambda^{1/2}} L(w)\, ds,
 &&
 g_3(t)= \int_0^t e^{is\Lambda^{1/2}} N_1(u,w)\, ds,\\
 g_4(t)&= \int_0^t e^{is\Lambda^{1/2}} N_2(w,w)\, ds, &&
 g_5(t)= \int_0^t e^{is\Lambda^{1/2}} N_3(u,w)\, ds,
 \label{}
\end{align}

\section{Elliptic estimates and the regularity of the free boundary}
\label{ellsec}

Much of the material in the following sections is based heavily on the
estimates and ideas in \cite{Christodoulou2000}.
In \cite{Christodoulou2000}, the authors
consider the free boundary problem for a bounded fluid region, but extending
their approach to the case of an unbounded domain is straightforward.

It is convenient to work in terms of
Lagrangian coordinates, which we now define.
We let $\Omega$ denote the lower half-plane in $\R^3$.
In this section, we will use the convention that points in
$\D_t$ are denoted by $x$ and points in $\Omega$ are denoted by
$y$.
The Lagrangian coordinates
$x(t) : \Omega \to \D_t$ are then defined by:
\begin{align}
  \frac{d}{dt} x_i(t, y) &= v_i(t, x(t,y))\, && y \in \Omega,\\
  x(0, y) &= y.
 \label{}
\end{align}
In these coordinates, the material derivative $D_t =
\pa_t + v^k\pa_k$ becomes
the usual time derivative:
\begin{align}
 D_t = \frac{\pa}{\pa t} \bigg|_{y = const.} = \frac{\pa}{\pa t}\bigg|_{x = const.}
  + v^k \frac{\pa}{\pa x_k}.
 \label{}
\end{align}

The Lagrangian coordinates $x$ induce a time dependent (co)metric $g$ on $\Omega$:
\begin{align}
 g_{ab} &= \delta_{ij} \frac{dx^i}{dy^a} \frac{dx^j}{dy^b},
 \quad g^{ab} = \delta^{ij} \frac{dy^a}{dx^i} \frac{dy^b}{dx^j}
 \label{metric}
\end{align}
We use the convention
that indices $a,b,c...$ denote quantities expressed in Lagrangian
coordinates and indices $i,j,k,..$ denote quantities expressed in
the $x$ coordinates.
We let $\na$ denote the covariant derivative
on $\Omega$ with respect to the metric $g$. We write
$\Gamma^a_{bc}$ for the Christoffel symbols:
\begin{equation}
 \Gamma^c_{ab} = \frac{1}{2} g^{cd}\bigg(
 \frac{\pa}{\pa y^a} g_{bd} + \frac{\pa}{\pa y^b} g_{ad}
 - \frac{\pa}{\pa y^d} g_{ab}\bigg),
 \label{}
\end{equation}
and the covariant derivative of a $(0,r)$ tensor $\beta$ is then:
\begin{align}
 \na_a \beta_{a_1\cdots a_r}
 = \pa_{y^a} \beta_{a_1\cdots a_r}
 - \Gamma_{aa_1}^d \beta_{d a_2 \cdots a_r}
 -\cdots \Gamma_{a a_r} \beta_{a_1\cdots a_{r-1} d}.
 \label{}
\end{align}

We let $d = d(t, p) = \dist_g(p, \pa \Omega)$ denote the geodesic
distance with respect to the metric $g$ from $p \in \Omega$ to $\pa \Omega$,
and we define the unit normal to $\pa \Omega$ by:
\begin{align}
 \n_a = \pa_a d, && \n^a = g^{ab}\n_b.
 \label{}
\end{align}
We will also write $\n_i$ for the normal expressed in Eulerian coordinates:
\begin{align}
 \n_i = \frac{\pa y^a}{ \pa x^i} \n_a, &&
 \n^i = \delta^{ij} \n_j
 \label{}
\end{align}

We let $\iota_0 = \iota_0(t)$ denote the injectivity radius
of $\pa \D_t$
By definition, this is the largest number $\iota_0$ so
that the map:
\begin{align}
 (x,\iota) \to x + \iota \n(x), \quad x \in \pa \D_t
 \label{}
\end{align}
is injective from $\pa \D_t \times (-\iota_0, \iota_0) \to \{ x \in \D_t :
d(t,p) < \iota_0\}.$

The (co)metric on $\pa \Omega$ is given by:
\begin{align}
 \gamma_{ab} =g_{ab} - \n_a \n_b,
 && \gamma_a^b = \delta_a^b - \n_a \n^b,
 \label{}
\end{align}
and
the second fundamental form of $\pa \Omega$ is:
\begin{align}
 \theta_{ab} = \gamma_a^c \gamma_b^d \nabla_c \n_d.
 \label{}
\end{align}
We note that on $\pa \Omega$, if $\tn$ denotes the
covariant derivative on $\pa \Omega$ with respect to the
metric $\gamma$, then:
\begin{equation}
 \tn_a \beta_{a_1\cdots a_r} = \gamma_a^b \gamma_{a_1}^{b_1}\cdots
 \gamma_{a_r}^{b_r} \na_b \beta_{b_1\cdots b_r}.
 \label{}
\end{equation}
In particular this implies that if $q$ is a function on $\Omega$ with
$q = 0$ on $\pa \Omega$ then $\gamma_a^b\na_b q = 0$ on $\pa \Omega$.

\subsection{The extension of the normal to the interior}

Since $d$ is the geodesic distance, we have
$\na_{\na d} \na d = 0$ and so $\na \na d = \tilde{\theta}$,
where $\tilde{\theta}$ is the second fundamental form for the surfaces
$\{d = \textrm{const}\}$. We will also write $\theta$ for the
second fundamental form of $\pa \Omega$; if $n_a$ is the unit normal
vector to $\pa \Omega$, then:
\begin{equation}
 \theta_{ab} = (\delta^c_a - n_a n^c)(\delta^d_b - n_bn^d) \na_c n_d.
 \label{}
\end{equation}

We now define an extension of the normal to a neighborhood
of the boundary.
We fix $d_0$ with $\iota_0/16 \leq d_0 \leq \iota_0/2$ and let
$\eta \in C^\infty(\R)$ be a function with $\eta(s) = 1$ when
$|s| \leq 1/2, \eta(s) = 0$ when $|s| \geq 3/4$, $0 \leq \eta(s) \leq 1$
and $|\eta'| \leq 4$. We then define:
\begin{align}
 \tilde{n}_a(p) = \eta\bigg( \frac{d(p)}{d_0}\bigg) \na_a d (x,y).
 \label{tilden}
\end{align}
Close to the boundary, we have $\tilde{n}_a = \na_a d$ and away from the
boundary, $\tilde{n}_a = 0$.
We will not need the following lemma explicitly but it is useful to note
that we can control
the regularity of $\widetilde{n}$. See Lemma 3.10 in \cite{Christodoulou2000}
for the proof.
\begin{lemma}
  \label{geomlem}
 With the above definitions, for each $y \in \pa \Omega$,
 if $d \leq \iota_0/2$:
 \begin{align}
  |\na \tilde{n}(q, d)| \leq 2 |\theta(q)|, &&
  |D_t \tilde{n}(q,d)| \leq 6 ||h||_{L^\infty(\Omega)},
  \label{}
 \end{align}
 where $h_{ab} = \frac{1}{2} D_t g_{ab}$.
\end{lemma}

We now extend $\gamma$ to the interior $\Omega$. Abusing
notation, we will write:
\begin{align}
 \gamma_{ab} = g_{ab} - \tilde{n}_a \tilde{n}_b,
 &&
 \gamma_a^b = \delta^{ac} \gamma_{bc}
 && \gamma^{ab} = g^{ac}g^{bd} \gamma_{cd}.
 \label{gammadef}
\end{align}
On $\pa \Omega$, $\gamma_{ab}$ (resp. $\gamma^{ab})$ is just the metric
(resp. cometric) on $\pa \Omega$
induced by $g$,
and $\gamma^a_b$ is the projection to $T(\pa \Omega)$. Away from
$\pa \D_t$, $\gamma_{ab} = g_{ab}$ and $\gamma^a_b$ is the identity
map.
The estimates in Lemma \ref{geomlem} then imply (see Lemma 3.11 in
\cite{Christodoulou2000}):
\begin{lemma}
 \label{geomlem2}
 With the above definitions, we have:
 \begin{align}
  |\na \gamma| \leq  C \bigg( ||\theta||_{L^\infty(\Omega)}
  + \frac{1}{\iota_0} \bigg),
  &&
  |D_t \gamma| \leq C ||h||_{L^\infty(\pa \Omega)}
  \label{}
 \end{align}
\end{lemma}

\subsection{Elliptic estimates}
\label{ellipticsec}

For notational convenience, in this section we write $x_3 = y$. We will use
multi-index notation and write $I = (i_1,\dots, i_r)$.
We will write $\na^r$ for the operator which has components:
\begin{equation}
 \na^r_{I} = \na_{i_1}\cdots \na_{i_r},
 \label{}
\end{equation}
If $i_j = 1,2$ for each $j = 1,..., r$, we will also write $\nabla^r$ for the operator:
\begin{equation}
 \nabla^r_I = \nabla_{i_1} \cdots \nabla_{i_r}.
 \label{}
\end{equation}
We will also write:
\begin{equation}
 \gamma^I_J = \gamma^{i_1}_{j_1}\cdots \gamma^{i_r}_{j_r},
 \label{}
\end{equation}
Let $\beta$ be a $(0,r+1)$ tensor with
$\beta_{i_1\cdots i_r i} = \na^r_{i_1 \cdots i_r} \alpha_i$ for some $(0,1)$-tensor $\alpha$. We write:
\begin{align}
  (\div \beta)_{I} &= \delta^{ij} \na _j\beta_{I} =
  \na^r_{I}( \delta^{ij}\na_j \alpha_i),\\
  (\curl \beta)_{ij} &= \na_i \beta_{I j} -
  \na_j \beta_{I i}
  = \na_{I}^r (\na_i \alpha_j - \na_j \alpha_i).
 \label{}
\end{align}
We will also write:
\begin{align}
 (\Pi \beta)_I = \gamma^{J}_I \beta_{J},\\
  (\n \cdot \beta)_I = \n^i \beta_{Ii}
 \label{}
\end{align}

We will rely heavily on the following pointwise estimate
in $\D_t$, which is originally from \cite{Christodoulou2000}:
\begin{lemma}
  If $\beta$ is as above, then:
 \begin{align}
  |\na \beta|^2
  \leq C \big( \delta^{ij}\gamma^{k\ell}\gamma^{IJ}(\na_k \beta_{Ii})
  (\na_\ell \beta_{Jj})
  + |\div \beta|^2 + |\curl \beta|^2\big), && \textrm{ in } \D_t,
  \label{pw}
 \end{align}
\end{lemma}

We will also use the following $L^2$ estimates:
\begin{lemma}
  \label{bigellprop}
 With the above notation, if
 $|\theta| + \frac{1}{\iota_0} \leq K$ then:
 \begin{align}
  ||\beta||_{L^p(\pa \Omega)}^p &\leq C \big( ||\na \beta||_{L^p(\Omega)}
  + K ||\beta||_{L^p(\Omega)}\big), \quad 1 < p < \infty,
  \label{trace}\\
  ||\beta||_{L^2(\pa \Omega)}^2 &\leq C ||\Pi \beta||_{L^2(\pa \Omega)}^2
  + C\big( ||\div \beta||_{L^2(\Omega)} +  ||\curl \beta||_{L^2(\Omega)}
  + K||\beta||_{L^2(\Omega)}\big) ||\beta||_{L^2(\Omega)},\label{traceproj}
  \\
  ||\beta||_{L^2(\pa \Omega)}^2 &\leq C ||\n\cdot\beta||_{L^2(\pa \Omega)}^2
  + C\big( ||\div \beta||_{L^2(\Omega)} +  ||\curl \beta||_{L^2(\Omega)}
  + K||\beta||_{L^2(\Omega)}\big) ||\beta||_{L^2(\Omega)},\label{traceproj2}
\end{align}
and
\begin{align}
  ||\na \beta||_{L^2(\Omega)}^2
  &\leq C ||\na \beta||_{L^2(\pa\Omega)} ||\beta||_{L^2(\pa \Omega)}
  + C \big(||\div \beta||_{L^2(\Omega)} + ||\curl\beta||_{L^2(\Omega)}\big)^2,
  \label{ezbdy}\\
  ||\na \beta||_{L^2(\Omega)}^2
  &\leq C ||\Pi \na \beta||_{L^2(\pa \Omega)} ||\Pi \n\cdot \beta||_{L^2(\pa \Omega)}
  + C \big(||\div \beta||_{L^2(\Omega)} + ||\curl \beta||_{L^2(\Omega)} +  K
  ||\beta||_{L^2(\Omega)}\big)^2,\label{bdyproj1}\\
  ||\na \beta||_{L^2(\Omega)}^2
  &\leq C ||\Pi \n \cdot \nabla \beta||_{L^2(\pa \Omega)} ||\Pi \beta||_{L^2(\pa \Omega)}
  + C \big(||\div \beta||_{L^2(\Omega)} + ||\curl \beta||_{L^2(\Omega)} +  K
  ||\beta||_{L^2(\Omega)}\big)^2.
  \label{bdyproj2}
 \end{align}
\end{lemma}

\begin{proof}
 Other than \eqref{trace} for $p \not =2$, all of the above inequalities
 are in Lemma 5.6 in \cite{Christodoulou2000}. To prove \eqref{trace} for
 $p \not= 2$ we can argue in essentially the same way as the $p = 2$ case;
 by Stokes' theorem:
 \begin{align}
  ||\beta||_{L^p(\pa \Omega)}^p
  = \int_{\pa \Omega} \tilde{n}_i \tilde{n}^i |\beta|^p\, dS
  &= \int_{\Omega} (\na_i \tilde{n}^i)|\beta|^p
  + p\nabla \beta \cdot \beta |\beta|^{p-2}.
  \label{}
 \end{align}
 By Lemma \ref{geomlem}, the first term is bounded by $K ||\beta||_{L^p(\Omega)}^p$. To
 bound the second term, we just note that by Holder's inequality
 and Young's inequality, it is bounded by $||\nabla \beta||_{L^p(\Omega)}
 ||\beta||_{L^p(\Omega)}^{p-1} \lesssim ||\nabla \beta||_{L^p(\Omega)}^p
 + ||\beta||_{L^p(\Omega)}^p$.

\end{proof}

The estimates \eqref{pw} will be used to show that the energy
(defined in \eqref{endef}) controls all derivatives
of $v$. The estimates in \eqref{bigellprop} will be used to show that
the energies control $v$ on the boundary, and we will also use them
with $\alpha = \nabla q$ for a function $q$ to
control solutions of the Dirichlet problem.
We will assume in many of the following estimates that $\K \leq 1$. This is only
for notational convenience and is not essential to the arguments;
many of the estimates will
involve constants which can be bounded in terms of $1 + \K$ and so this assumption
allows us to ignore the unimportant dependence on $\K$. We will make it clear when
this assumption is used. Versions of these estimates with more explicit
dependence on $\K$ can be found in \cite{Christodoulou2000}.

 First, we show that derivatives
of $q$ can be controlled by projected derivatives of $q$ on the boundary
and derivatives of $\Delta q$:
\begin{prop}
  If $\K \leq 1$ then for $r \geq 1$:
 \begin{equation}
  ||\na^r q||_{L^2(\pa \Omega)} + ||\na^r q||_{L^2(\Omega)}
  \leq C \bigg( ||\Pi \na^r q||_{L^2(\pa \Omega)} +
   \sum_{s \leq r-1}
  ||\na^s \Delta q||_{L^2(\Omega)} + ||\na q||_{L^2(\Omega)}\bigg),
  \label{mainell}
 \end{equation}
 and for any $\delta > 0$:
 \begin{equation}
  ||\na^r q||_{L^2( \Omega)} + ||\na^{r-1} q||_{L^2(\pa \Omega)}
  \leq \delta ||\Pi \na^r q||_{L^2(\pa \Omega)} + C(1/\delta)
  \sum_{s \leq r-2}
  ||\na^s \Delta q||_{L^2(\Omega)} + ||\na q||_{L^2(\Omega)}.
  \label{mainell2}
\end{equation}
\end{prop}

\begin{proof}
  By \eqref{traceproj} with $\beta = \na^r q$:
 \begin{equation}
  ||\na^r q||_{L^2(\pa \Omega)}^2
  \leq ||\Pi \na^r q||_{L^2(\pa \Omega)}^2
  + C\big( ||\na^{r-1}\Delta q||_{L^2(\Omega)} + \K||\na^r q||_{L^2(\Omega)}
  \big) ||\na^r q||_{L^2(\Omega)},
  \label{}
 \end{equation}
 and by \eqref{bdyproj1} with $\beta = \na^{r-1}q$:
 \begin{equation}
  ||\na^r q||_{L^2(\Omega)}^2
  \leq C ||\Pi \na^r q||_{L^2(\pa \Omega)}
  ||\na^{r-1} q||_{L^2(\pa \Omega)}
  + C \big(||\na^{r-1} \Delta q||_{L^2(\Omega)}
  + \K ||\na^{r-1} q||_{L^2(\pa \Omega)}\big)^2.
  \label{}
 \end{equation}
 Combining these inequalities and using induction
 gives \eqref{mainell} and \eqref{mainell2}.
\end{proof}

We will use this proposition in two ways. First, in our energy estimates
we will directly control $||\Pi \na^r p||_{L^2(\Omega)}$ if the Taylor
sign condition \eqref{tsc} holds and since $\Delta p = -(\pa_i v^j)(\pa_j v^i)$,
we control this as well. We will also use this estimate to control
derivatives $D_t p$ on
$\pa \D_t$, and we will rely on the observation that
$\Pi \na^r q$ is lower order
if $q = 0$ on $\pa \Omega$. This is clear when $r = 0,1$, and for $r = 2$
we have:
\begin{equation}
 \Pi_i^j\Pi_k^\ell \na_{j}\na_\ell q =
 \Pi_i^j \na_j \big( \Pi_k^\ell \na_\ell q\big) - \Pi_i^j \na_j(\Pi_k^\ell)
 \na_\ell q,
 \label{basicproj}
\end{equation}
and when $q = 0$ on $\pa \Omega$, the first term is zero and the second term is
$-(\Pi_i^j \na_j \n_k)\n^\ell\nabla_\ell q$, so that $\Pi\nabla^2 q = \theta \na_\n q$.
We also record the $r = 3$ case for later use:
\begin{equation}
 \Pi \na^3 q = \tn^3 q - 2 \theta \otimes (\theta \cdot \tn q)
 + (\tn \theta) \na_N q + 3 \theta \otimes (\tn \na_N q).
 \label{3rdorder}
\end{equation}
It will not be important in our argument exactly which indices appear where.

One can
use the following heuristic argument from
\cite{Christodoulou2000} to see what the higher-order version of the
formula is. If $d(x) = \dist(x, \pa \Omega)$ then $q/d$ is smooth up to the boundary,
and:
\begin{equation}
 \Pi \na^r q = \Pi \na^r \bigg(   d\frac{q}{d}\bigg)
 = \sum_{s = 0}^{r} \Pi (\na^s d)\otimes \na^{r-s}\bigg(\frac{q}{d}\bigg).
 \label{}
\end{equation}
Restricting this formula to the boundary, we see that the $s = 0, 1$ terms
drop out and that $q / d \sim \nabla_\n q$. Further, if we knew that all the
derivatives falling on $d$ were purely tangential, then arguing
as above we could replace
$\na^s d$ with $\na^{s-2} \theta$. We therefore write $\na_i = (\Pi_i^j+ \n_i \n^j)\na_j$
and further note that $\n^i \n^j \na_j \na_k d = 0$ because $d$ is the geodesic
distance.
Each time we make this subsitution, some derivatives will fall onto the
factors of $N$ we have introduced and this generates more factors
of $\theta$, but at the same time less derivatives land on the function
$q$.
This suggests that we should expect:
\begin{equation}
 \Pi \na^r q
 \sim \sum_{s = 0}^{r-2} \tn^s \theta \otimes \na^{r-s} \na_\n q
 \label{heuristic}
\end{equation}
Also note that the
$s = r-2$ term of the expansion \eqref{heuristic} is $(\tn^{r-2} \theta) \na_\n q$
and so if the lower order terms and $|\na q|^{-1}$ are bounded,
this gives an estimate for $\theta$ in terms of
$q$.

The rigorous version of these observations is:
\begin{prop}
  \label{projlem}
  Let $q : \D_t \to \R$ be a function.
  If $||\theta||_{L^\infty(\pa \Omega)} \leq 1$, then for $m = 0, 1$:
\begin{align}
 ||\Pi \na^r q||_{L^2(\pa \Omega)}^2 & \leq ||\tn^r q||_{L^2(\pa\Omega)} + 2 ||\tn^{r-2} \theta||_{L^2(\pa \Omega)} \,||\na_\n q||_{L^\infty(\pa \Omega)}\\
 & +C\bigg( ||\theta||_{L^\infty(\pa \Omega)}
 + \sum_{k \leq r-2-m} ||\tn^k \theta||_{L^2(\pa \Omega)}\bigg)
 \sum_{k \leq r-2+m} ||\na^k q||_{L^2(\pa \Omega)}
 +
  C \sum_{k = 1}^{r-1} ||\na^{r-k}
 q||_{L^2(\pa \Omega)}\label{projell1}
\end{align}
and if $|\na_\n q| > \delta_0 > 0$:
\begin{multline}
 ||\tn^{r-2} \theta||_{L^2(\pa \Omega)} \leq
 C\delta_0^{-1} \bigg(
||\Pi \na^r q||_{L^2(\pa \Omega)} + \sum_{k = 1}^{r-1}
||\na^{r-k} q||_{L^2(\pa \Omega)}
  \bigg)\\
+ C\delta_0^{-1}\bigg( ||\theta||_{L^\infty(\pa \Omega)}
+ \sum_{k \leq r-3} ||\tn^{r-3} \theta||_{L^2(\pa \Omega)}
\bigg) \sum_{k \leq r-1} ||\na^k q||_{L^2(\pa \Omega)}
 \label{projell2}
\end{multline}
\end{prop}

Combining these two propositions, we have:
\begin{cor}
 If $\K \leq 1$ and $q: \D_t \to \R$ is a function with $q = 0$ on $\pa \Omega$,
 then for $r \geq 3$:
 \begin{multline}
  ||\na^{r-1} q||_{L^2(\pa \Omega)} \leq C
  \Big( ||\tn^{r-3}\theta||_{L^2(\pa \Omega)} ||\na_\n q||_{L^\infty(\pa \Omega)}
  + ||\na^{r-2} \Delta q||_{L^2(\Omega)}\\
  + C(||\theta||_{L^2(\pa \Omega)}, ...,||\tn^{r-4} \theta||_{L^2(\pa \Omega)})
  \Big( ||\na_\n q||_{L^\infty(\pa \Omega)} +
  \sum_{s \leq r-3} ||\na^s \Delta q||_{L^2(\Omega)} + ||\na q||_{L^2(\Omega)}
  \Big),
  \label{}
 \end{multline}
 and for $r > 3$:
 \begin{multline}
  ||\na^{r-1} q||_{L^2(\pa \Omega)} + ||\na q||_{L^\infty(\pa \Omega)}
  \leq C ||\na^{r-2}\Delta q||_{L^2(\Omega)}
  + C(||\theta||_{L^2(\pa \Omega)},..., ||\tn^{r-3} \theta||_{L^2(\pa \Omega)})
  \sum_{s \leq r-3} ||\na^s \Delta q||_{L^2(\Omega)}.
  \label{}
 \end{multline}
\end{cor}

\subsection{Estimates for $v_\omega$ }

Unlike the previous section, in this section we will work
on $\D_t$. We will therefore write:
\begin{align}
 \gamma_{ij} = \frac{\pa y^a}{\pa x^i} \frac{\pa y^b}{\pa x^j}\gamma_{ab},
 \label{}
\end{align}
and similarly for $\gamma_i^j, \na_i,$ etc.
In Section \ref{ellsyssec}, we use some of the ideas from
\cite{Cheng2017} to show that $v_\omega = \curl \beta$, where $\beta$
satisfies:
\begin{alignat}{2}
 \Delta \beta & = \omega && \quad\textrm{ in } \D_t,\label{alpha1}\\
 \gamma_i^j \beta_j &= 0, &&\quad j = 1,2,3,\textrm{ on } \pa \D_t,\label{alpha2}\\
 \na_\n (\beta\cdot \n) + H \beta\cdot \n &=  0 &&\quad \textrm{ on } \pa\D_t.
 \label{alpha3}
\end{alignat}
Taking the divergence of \eqref{alpha1} and noting that
$\na \cdot \beta|_{\pa \D_t} = \gamma\na \cdot \na(\gamma \cdot \beta)
+ \na_\n(\beta\cdot \n) + H \beta\cdot \n = 0$, it follows that
$\div \beta = 0$ in $\D_t$ if $\beta$.
We have the basic elliptic estimate:
\begin{lemma}
  With $\beta$ as defined above:
\begin{align}
 ||\beta||_{L^6(\D_t)} + ||\na \beta||_{L^2(\D_t)}
 + ||\curl \beta||_{L^2(\D_t)}
 \lesssim ||\omega||_{L^{6/5}(\D_t)}.
 \label{alphabd}
\end{align}
\end{lemma}
\begin{proof}
  First, by the Sobolev inequality \eqref{sob1},
   $||\beta||_{L^6(\D_t)}
  \lesssim ||\na \beta||_{L^2(\D_t)}$. We next show that
  $||\na \beta||_{L^2(\D_t)} \lesssim ||\curl \beta||_{L^2(\D_t)}$.
  Note that:
 \begin{equation}
  \int_{\D_t} \delta^{ij}\delta^{k\ell}
  \na_i \beta_k \na_j \beta_\ell
  = \int_{\D_t} \delta^{ij} \delta^{k\ell}
  \na_i \beta_k \na_\ell \beta_j +
  \int_{\D_t} \delta^{ij} \delta^{k\ell}
  \na_i\beta_k \curl \beta_{j\ell }.
  \label{step1}
 \end{equation}
 Integrating by parts, the first term is:
 \begin{equation}
  \int_{\pa\D_t} \delta^{ij}
  \n^k \na_i \beta_k \beta_j
  - \int_{\D_t} \delta^{ij}\delta^{k\ell}
  (\na_\ell \na_i \beta_k) \beta_j
  \label{}
 \end{equation}
 The interior term vanishes since $\div \beta = 0$.
 To handle the boundary term, we note that since $\gamma\cdot \beta = 0$
 on $\pa \D_t$:
 \begin{equation}
  \n^k\beta^i \na_i \beta_k
  = \n^k \n^i (\beta^\ell \n_\ell)\na_i \beta_k
  = \n^i (\beta^\ell \n_\ell) \na_\n (\n^k\beta_k)
  -\n^i (\beta^\ell \n_\ell) H \n^k\beta_k
  = \big(\div \beta - \gamma^i_j \na_i (\gamma^k_\ell \beta^\ell)\big) (\beta^\ell \n_\ell)
  = 0,
  \label{}
 \end{equation}
where we have used that $\div \beta = 0$.
Returning to \eqref{step1}, we have:
\begin{equation}
 ||\na \beta||_{L^2(\D_t)}^2 \lesssim ||\na \beta||_{L^2(\D_t)}
 ||\curl \beta||_{L^2(\D_t)},
 \label{}
\end{equation}
which implies the bound for $||\na \beta||_{L^2(\D_t)}$.

Finally we show that $||\curl \beta||_{L^2(\D_t)} \lesssim
||\omega||_{L^{6/5}(\D_t)}$. Integrating by parts:
\begin{equation}
 \int_{\D_t} |\curl \beta|^2
 = \int_{\pa \D_t} (\n \times \beta) \curl \beta
 - \int_{\D_t} \beta \curl^2 \beta.
 \label{}
\end{equation}
Since the tangential components of $\beta$ vanish on $\pa \D_t$, it follows
that $\n\times \beta = 0$. The
interior term is bounded by $||\beta||_{L^6(\D_t)} ||\omega||_{L^{6/5}(\D_t)}$,
which completes the proof.
\end{proof}

The above estimates combined with the elliptic estimates in the previous
section will allow us to bound $||v_\omega||_{H^{r}(\D_t)}$. In the proof
of the dispersive estimates, we will also need to bound
$||V_\omega||_{L^p(\pa \D_t)}$
for $1 < p < 2$. Recall that in the interior, we have $V_\omega = \curl \beta$
with $\Delta \beta = \omega$. In the flat case ($h = 0$), a simple calculation
using the Newtonian potential shows that for any $z \in \{ (z_1, z_2, z_3)| z_3 \leq 0\}$,
we have
$|v_\omega(z)| = |\curl \beta(z)| \lesssim \frac{1}{1+|z|^2}||(1 + |z|^2)\omega||_{L^1(\D_t)}$.
Restricting
this to $z = (x, 0) \in \pa \D_t$ gives that $V_\omega \in L^p(\pa \D_t)$ for
$p > 1$. To handle the case with $h \not=0$, in
Proposition \ref{divcurllem}, we follow the approach of
\cite{Hofmann2007} \cite{Gruter1982} to
construct a Green's function for $\D_t$ which satisfies
the same estimates as the Newton potential, and this can be used to prove estimates
for $||V_\omega||_{L^p(\pa \D_t)}$ for $1 < p$.

\begin{prop}
  \label{vomegabd}
  If $||h||_{W^{4,\infty}(\R^2)} + ||h||_{H^{N_1}(\R^2)} \leq 1$,
  then for $2 \leq p < \infty$, $0 \leq r \leq N_1 - 2$:
 \begin{equation}
  ||\nabla^r V_\omega||_{L^p(\R^2)}
  + ||\na^{r} v_\omega||_{L^2(\D_t)}
  \lesssim  ||\omega||_{H^{N_1}_w(\D_t)},
  \label{mainomegabd}
 \end{equation}
 and for $1 < p \leq 2$:
 \begin{equation}
   ||V_\omega||_{L^p(\R^2)} \lesssim ||\omega||_{H^{N_1}_w(\D_t)}
   \label{lowp}
 \end{equation}
\end{prop}
The assumption on the size of $h$ is for notational convenience and can be
avoided. We remark that by the interpolation inequality \eqref{interpolationgms},
if the bootstrap assumptions \eqref{bootstrap1}-\eqref{bootstrap3} hold
then we have:
\begin{equation}
 ||h||_{H^{N_1}(\R^2)} \lesssim \ve_0 (1 + t)^{\sigma} + \ve_1(1+t)^{\delta},
 \label{}
\end{equation}
where $\sigma = \frac{N_1 + 1}{N_0 - 1}(1 + \delta)$ and so this quantity
is less than one until $t \sim T_{\ve_0, \ve_1}$. We will be forced to take
$t \lesssim T_{\ve_0, \ve_1}$ at other points anyways, so
this is not a serious restriction.

\begin{proof}
  First, by \eqref{cr}, $\nabla^r V_\omega = (\na^r v_\omega)|_{\pa \D_t}
  - \nabla^r h (\na_y v_\omega)|_{\pa \D_t}
  + (\nabla h)^r (\na_y^r v_\omega)|_{\pa \D_t} + ... $ , up to similar
  terms. We show how to prove the estimates for the first term, as
  the other terms can be handled similarly.
  We consider the cases $1 < p < 2$ and $p \geq 2$ separately.

  When $p \geq 2$, by Holder's, Young's and Sobolev's inequalities,
  it suffices to control $||\na^k V_\omega||_{L^2(\D_t)}$ for
  $0 \leq k \leq r+2$. Since $v_\omega \cdot \n = 0$ on $\pa \D_t$,
  repeatedly applying
  the trace inequality \eqref{traceproj} gives:
  \begin{equation}
   ||\na^k v_\omega||_{L^2(\pa \D_t)}^2 \leq
   C \Big( ||\na^{k} \omega ||_{L^2(\D_t)}^2 + (1 + K) ||v_\omega||_{L^2(\D_t)}^2\Big).
   \label{}
  \end{equation}
The constant here depends on bounds for $||\theta||_{L^\infty(\pa \D_t)}$
as well as $||\theta||_{H^{k-2}(\pa \D_t)}$ and by assumption these are both
bounded.
  By the estimate \eqref{alphabd}, we have $||v_\omega||_{L^2(\D_t)} \lesssim
  ||\omega||_{L^{6/5}(\D_t)}$ and by Holder's inequality, we have
  $||\omega||_{L^{6/5}(\D_t)} \lesssim ||\omega||_{H^{N_1}_w(\D_t)}$.
  Since $k \leq r+ 2 \leq N_1$ we bound the first term here as well.

The estimate \eqref{lowp} follows from \eqref{decay}.

\end{proof}

\section{Energy Estimates}
\label{energyestimates}
The system \eqref{mom}-\eqref{freebdy} has a conserved energy:
\begin{equation}
 E_0(t) = \frac{1}{2}\int_{\D_t} |v(t,x,y)|^2 dxdy +
 \frac{1}{2}\int_{\R^2} |h(t,x)|^2 \, dx
 = \frac{1}{2} \int_{\R^2} \int_{-\infty}^{h(t,x)} |v(t,x,y)|^2\, dx dy
 + \frac{1}{2} \int_{\R^2} |h(t,x)|^2\, dx.
 \label{conserved}
\end{equation}
We have:
\begin{align}
 \frac{d}{dt} E_0(t) &= \int_{\D_t}  v^i \pa_t v_i \,dxdy + \frac{1}{2}\int_{\R^2}
 \pa_t h |v|^2 \,dx +  \int_{\R^2} h \pa_t h\, dx\\
 &= -\int_{\D_t} v^i (v^k\pa_k v_i + \pa_i (p +y) )\,dxdy
 + \frac{1}{2}\int_{\R^2} \pa_t h |v|^2 \,dx+
 \int_{\R^2} h \pa_t h\, dx\\
 &= -\frac{1}{2} \int_{\pa\D_t} n_k v^k |v|^2\, dS - \int_{\pa \D_t} n_i v^i h
 \,dS + \frac{1}{2}\int_{\R^2}  \pa_t h |v|^2\, dx + \int_{\R^2}
 h \pa_t h\, dx,
 \label{}
\end{align}
where we used that $\div v = 0$ in $\D_t$ and that $p = 0$ on $\pa \D_t$. Using
\eqref{freebdy2} the first and third, and second and fourth terms here cancel.

To get higher-order energies, in the irrotational case ($\omega = 0$) one
can use the system \eqref{hamil1}-\eqref{hamil2} directly to prove energy estimates.
See \cite{Germain2012} or \cite{Alazard2014} for this approach.
In the case $\omega \not =0$, the corresponding system
\eqref{whamil1}-\eqref{whamil2} is more complicated to work with and we instead choose to model our approach on
\cite{Christodoulou2000} and prove energy estimates for Euler's equation
\eqref{mom}
-\eqref{freebdy} directly. The advantage is that the estimates can be proved
using
elementary techniques, relying only on integration by parts
and simple geometric facts
(such as \eqref{heuristic}, \eqref{traceproj}).

We define the projection $\gamma$ as in \eqref{gammadef}. We will write:
\begin{equation}
 \gamma^{ij} = \frac{\pa x^i}{\pa y^a} \frac{\pa x^j}{\pa y^b}
 \gamma^{ab},
 \label{}
\end{equation}
for $\gamma$ expressed in the $x$-coordinates. We also write:
\begin{align}
 \gamma^{i_1\cdots i_r j_1 \cdots j_r} =
 \gamma^{i_1j_1} \cdots \gamma^{i_r j_r}
 \label{}
\end{align}
For $(0,r)-$tensors $\alpha, \beta$, we define:
\begin{equation}
 Q(\alpha, \beta) = \gamma^{i_1 \cdots i_r j_1 \cdots j_r}
 \alpha_{i_1\cdots i_r} \beta_{j_1 \cdots j_r}.
 \label{}
\end{equation}
The energies are:
\begin{equation}
 \E^r(t) = \int_{\D_t} \delta^{ij} Q(\na^r v_i, \na^r v_j)\, dV
 + \int_{\pa \D_t} Q(\na^r p, \na^r p)|\na p| \, dS
 + \int_{\D_t} |\na^{r-1} \omega|^2 \, dV.
 \label{endef}
\end{equation}

We will see that since $p = 0$ on $\pa \D_t$, $Q(\na^r p, \na^r p)
= Q(\tn^{r-2}\theta, \tn^{r-2} \theta)|\na_\n p|^2$ to highest order
(see the discussion after \eqref{basicproj} and the estimate \eqref{coertheta}).
In particular since $\theta \sim \nabla^2 h$, bounds for $\E^r$ imply bounds for
$h$. Moreover, in Theorem \ref{uenests}, we will see that $||u||_{H^{N_0-1}}^2
\lesssim \E^{N_0}$ to highest order (recall that $u$ is defined in
\ref{realudef}).

We will prove the energy estimates in the following sections assuming
the following a priori bounds:
\begin{alignat}{2}
 |\theta(t)| + \frac{1}{\iota_0(t)} &\leq K && \quad \text{ on } \pa \D_t,
 \label{Kbd}\\
 -\nabla_\n p(t) \geq \delta_0 &> 0  && \quad \text{ on } \pa \D_t,\label{tscbd}\\
 |\na^2 p(t)| + |\na_\n D_t p(t)| &\leq L  && \quad \text{ on } \pa \D_t,
 \label{d2pbd}\\
  |\na v(t)| + |\na^2 p(t)| &\leq M && \quad \text{ on } \D_t.
  \label{vbd}
\end{alignat}
Recall that we are writing $\iota_0(t)$ for the injectivity radius of
$\pa \D_t$.
We will assume in the estimates that $\K \leq 1$. This is only
for notational convenience and is not essential to the arguments; many
of the estimates will involve coefficients that can be bounded in terms of
$1 + K$ and this allows us to ignore the unimportant dependence on $K$.
We also remark that $\frac{1}{\iota_0} \leq
||\theta||_{L^\infty(\pa \D_t)}$ and so the definition of $\K$ is somewhat
overcomplicated. We choose to keep track of both terms because
it turns out that if one is interested in proving energy estimates
which depend on as few derivatives of $v$ as possible in $L^\infty$, it is
difficult to control the time evolution of $\iota_0$. For this reason,
in \cite{Christodoulou2000}, the authors introduce another radius which
they denote $\iota_1$ (see Definition 3.5 there) which can
be used to control $\iota_0$. For our purposes this distinction
will not be important, because we will eventually need to assume bounds
for more derivatives of $v$ in any case, but if one is interested
in studying this problem with less regular data it is useful to keep track
of both terms.

The main result of this section is the following energy estimate:
\begin{prop}
  \label{enestthm}
 Suppose that the a priori assumptions \eqref{Kbd}-\eqref{vbd} hold.
 There are continuous functions $C_r = C_r(\delta_0^{-1})$ and
 homogeneous polynomials $P_r$ with positive coefficients so that for
 $r \geq 0$:
\begin{equation}
 \Big|\frac{d}{dt}\E_r(t)\Big|
 \leq C_r(\delta_0^{-1}) (K + L + M) \Big( \E_r(t) +
 (K + L + M)P_r(\E_{r-1}^*(t), K, L, M)\Big),
 \label{dten}
\end{equation}
with $\E_{r-1}^* = \sum_{s \leq r-1} \E_s$.
\end{prop}
We prove this in the next two subsections.
Next, we relate the energy $\E_r$ and the a
priori assumptions \eqref{Kbd}-\eqref{vbd} to the dispersive variable
$u$ and the vorticity.
\begin{lemma}
  \label{A1lem}
 If the bootstrap assumptions \eqref{bootstrap1}-\eqref{bootstrap3} hold, then
 with:
 \begin{equation}
  \A(t) =
   ||\theta(t)||_{L^\infty(\pa \D_t)} + \frac{1}{\iota_0(t)}
   + ||\na^2 p||_{L^\infty(\pa \D_t)} + ||\na D_t p||_{L^\infty(\pa \D_t)}
   +
  ||\na v(t)||_{L^\infty(\D_t)} + ||\na p(t)||_{L^\infty(\D_t)},
  \label{Adef}
 \end{equation}
 and
 \begin{equation}
  \B(t) = ||h(t)||_{W^{4,\infty}(\R^2)} +
  ||\varphi(t)||_{W^{4,\infty}(\R^2)}
  + ||\omega(t)||_{H^{N_1}_w(\D_t)},
  \label{}
 \end{equation}
 we have:
 \begin{equation}
  \A(t) \lesssim \B(t)
  \bigg(1 + \B(t)\sqrt{\E^*_3(t)}\bigg),
  \label{a1est}
\end{equation}
 where $\E^*_3 = \sum_{s \leq 3} \E_s$.
 Furthermore, if $0 \leq t \leq T_{\ve_0,\ve_1}$ with
 $T_{\ve_0, \ve_1}$ defined by \eqref{tdef}, then:
 \begin{equation}
  -\nabla_\n p(t) \geq \frac{1}{2} (-\nabla_\n p(0))
  \quad \text{ on } \pa \D_t.
  \label{tsest}
 \end{equation}
\end{lemma}
We remark that one could replace $||\omega||_{H^{N_1}_w(\D_t)}$
in \eqref{a1est} with
an $L^\infty$-based norm with fewer derivatives by using a Schauder
estimate, but this will suffice
for our purposes. We also note that the fact that $\sqrt{\E}_3$ shows up
on the right-hand side of \eqref{a1est} is because we need to control
$||\na D_t p||_{L^\infty(\pa \D_t)}$. We bound this by Sobolev embedding
and then the elliptic estimates in Section \ref{ellsec}. Since $\Delta D_t p$
is cubic (see \eqref{dtpeq}), this can be bounded by $\B^2\sqrt{\E_3^*}$.

Recall that $\varphi = \psi|_{\pa \D_t}$
where $\nabla_\n \psi = \n \cdot v$ on $\pa \D_t$. Since by
Lemma \ref{coerlem}, the energies control derivatives
of $v$ on $\pa \D_t$ as well as derivatives of $\theta$, we have the
following estimate, which is proved in Section
\ref{coerensec}.
\begin{prop}
  \label{uenests}
  With $\varphi_\omega$ defined by \eqref{varphiwdef},
  if $||h||_{W^{4,\infty}(\R^2)} \ll 1$, then for any $r \geq 1$:
  \begin{equation}
    ||h||_{H^r(\R^2)}^2 +
    ||\Lambda^{1/2} \varphi_\omega||_{L^2(\R^2)}^2
    + ||\nabla \varphi_\omega||_{H^{r-1}(\R^2)}^2
    \leq C  \E^r + \A P(\E^{r-1}_*, \A),
    \label{rhs1}
  \end{equation}
  where $\E^r_* = \sum_{s \leq r} \E^s$ and $\A$ defined by \eqref{Adef}.
\end{prop}

We will then see that the energy estimates \eqref{dten} and this
lemma imply:
\begin{prop}
  \label{uenlem}
If the bootstrap assumptions \eqref{bootstrap1}-\eqref{bootstrap3} hold,
then:
\begin{equation}
 ||\omega(t)||_{H^{N_1}_w(\D_t)}^2 \leq ||\omega_0||_{H^{N_1}_w(\D_0)}^2
 + C_{N} \int_0^t \bigg(||u(s)||_{W^{N_1 + 2,\infty}(\R^2)} +
 ||\omega(s)||_{H^{N_1}_w(\D_s)}\bigg)||\omega(s)||_{H^{N_1}_w(\D_s)}^2\, ds.
 \label{wenest}
\end{equation}
\end{prop}

 We will need to take $N_1 \geq 6$
to prove the dispersive estimates and since we only control
$||u||_{W^{4,\infty}}$, the result is that
$||u||_{W^{N_1 +2, \infty}(\R^2)}$ decays slightly slower than the critical
rate of $1/t$, and this is why we are only able to follow the solution
until $T \sim \ve_0^{-N}$.

Assuming these results for the moment, we can now provide the proofs
of Theorem \ref{intenprop} and \ref{intvortprop}:
\begin{proof}[Proof of Theorem \ref{intenprop}]
  By Gr\"{o}nwall's inequality and induction, if $\delta$ is sufficiently
  small then the estimate \eqref{dten} combined with \eqref{a1est} implies
  that there is a constant $C^E_{N_0}$ with:
  \begin{align}
   \E_{N_0}(t) &\leq \E_{N_0}(0) + C^E_{N_0}\int_0^t
   \big(\frac{\ve_0}{1+s} + \ve_1(1+s)^{\delta}\big) \ve_0^2(1+s)^{2\delta}
   \, ds
   \\ &\leq \E_{N_0}(0) + C^E_{N_0} (\ve_0^3(1+t)^{2\delta} +
   \ve_1\ve_0^2 (1+t)^{1+3\delta}).
   \label{}
  \end{align}
  Using \eqref{coertheta} and the fact that $\theta \sim \nabla^2 h$
  completes the proof.
\end{proof}

\begin{proof}[Proof of Theorem \ref{intvortprop}]

By the interpolation inequality \eqref{interpolationgms}
combined with the estimate \eqref{rhs1} for $||u||_{H^{N_0}}$,
\begin{align}
 ||u||_{W^{N_1 + 2,\infty}(\R^2)} \lesssim \frac{\ve_0}{(1+t)^{1-\sigma'}},
 \label{}
\end{align}
with $\sigma' = \frac{N_1+2}{N_0-1}(1+\delta)$, provided
the assumptions \eqref{bootstrap1}-\eqref{bootstrap3} hold. Recalling that
$N_0 = 2N N_1$, this implies:
\begin{equation}
 ||u||_{W^{N_1,\infty}(\R^2)}
 \lesssim \frac{\ve_0}{(1 + t)^{1-1/N}}.
 \label{}
\end{equation}
Combining this with \eqref{wenest} and using the assumptions
 \eqref{bootstrap1}-\eqref{bootstrap3}, we have:
 \begin{align}
  ||\omega(t)||_{H^{N_1}_w(\D_t)}^2
  &\leq ||\omega(0)||_{H^{N_1}_w(\D_0)}^2
  + C_{N_1}\int_0^t
\bigg( \ve_0 (1 +s)^{-1+\sigma} +
 \ve_1 ( 1+s)^{1+\delta}\bigg)
 \ve_1^2 (1+s)^{2\delta} \, ds\\
 &\leq
\frac{1}{4} \ve_1^2 +
 C_{N_1}
 \big( \ve_0 (1+t)^{\sigma} \big)
 \big(\ve_1 ( 1+ t)^{\delta}\big)
 \ve_1^2 (1 + t)^{2\delta},
  \label{}
 \end{align}
 as required.
\end{proof}

As in \cite{Christodoulou2000}, before proving the energy estimates \eqref{dten},
it is convenient to first prove that $\E^r$ controls norms of $v, p$ and the
second fundamental form $\theta$.

\subsection{Quantities controlled by $\E_r$}

We start with the equations for $\omega$ and $p$.
Taking the curl of \eqref{mom} shows that $\omega$
satisfies:
\begin{equation}
 D_t \omega_{ij} = \omega_{ik}\na^k v_j.
 \label{omegaeq}
\end{equation}
Taking the divergence of \eqref{mom} and using \eqref{mass}
gives that $p$ satisfies:
\begin{equation}
 \Delta p = -(\na_i v^j)(\na_j v^i)
 = -\na_i( v^j\na_j v^i),
 \label{peq}
\end{equation}
where we used that $\div v = 0$.
We will also need to use the equation for $D_tp$. We apply $D_t$ to
both sides of \eqref{peq}, and the right-hand side is:
\begin{equation}
- D_t (\na_i (v^j \na_j v^i))
= -\na_i\big( D_t (v^j \na_j v^i)\big) + \na_k
\big(v^k \na_i (v^j \na_j v^i)\big),
 \label{}
\end{equation}
while:
\begin{equation}
 D_t \Delta p = \na_i D_t \na^i p - \na_i v^k \na_k \na^i p
 = \Delta D_t p - \na_i \big( v^k\na_k \na^i p\big) - \na_k (\na_i v^k \na^i p).
 \label{}
\end{equation}
In particular, rearranging the indices this shows that:
\begin{equation}
 \Delta D_t p = \na_i \Big( v^k \na_k \na^i p - \na_k v^i \na^k p
 - D_t (v^j \na_j v^i) + v^i \na_k(v^j \na_j v^k)\Big).
 \label{dtpdiv}
\end{equation}
We shall need that the right-hand side of \eqref{dtpdiv} is
the divergence of a vector field, but for most of out applications
it is more useful to use \eqref{mass} and re-write this in the following
slightly more attractive way:
\begin{equation}
 \Delta D_t p = 4 \tr \big( (\na v) \cdot \na^2 p\big)
 + 2 \tr \big( (\na v)^3\big) - (\Delta v)\cdot \na p,
 \label{dtpeq}
\end{equation}
where we are writing $((\na v)\cdot \na^2 p)_{ij}
= \na_i v^k\na_{k}\na_j p$ and
$(( \na v)^3)_{ij} = \na_i v^k \na_k v^\ell \na_\ell v_j$.
The next lemma follows from these observations, the
interpolation inequalities \eqref{interp1}-\eqref{interp2},
and the fact that $[D_t, \pa_i] = -(\pa_i v^j)\pa_j$.
\begin{lemma}
  \label{coerlem}
  If $\K \leq 1$ then
  there are constants $C_r > 0$ so that:
 \begin{align}
  || D_t \na^r v + \na^{r+1} p||_{L^2(\D_t)}
  + ||D_t \na^{r-1} \omega||_{L^2(\D_t)} + ||\Delta \na^{r-1} p||_{L^2(\D_t)}
  \leq C_r ||\na v||_{L^\infty(\D_t)} \sum_{k = 0}^r ||\na^k v||_{L^2(\D_t)},
  \label{mainlot}
\end{align}
 \begin{align}
  ||\Pi \big( D_t \na^r p + (\na^r v)\cdot \na p
  - \na^r D_tp\big)||_{L^2(\pa \D_t)}
 \leq C_r \sum_{s = 1}^{r-2}
  || \Pi\big( (\na^{1+s} v) \cdot (\na^{r-s}p)\big)||_{L^2(\pa \D_t)}
  \label{projlot}
 \end{align}
 and
 \begin{multline}
  ||\na^{r-2} \Delta D_t p - (\na^{r-2} \Delta v)\cdot \na p
  ||_{L^2(\D_t)}\\
  \leq C_r \big( ||\na v||_{L^\infty(\D_t)}^2
  + ||\na p||_{L^\infty(\D_t)}\big)
  \bigg(\sum_{s = 1}^{r} ||\na^s v||_{L^2(\D_t)}^2 + ||\na^s p||_{L^2(\D_t)}
  \bigg)\\
  +C_r ||\na v||_{L^\infty(\D_t)}^2 \sum_{s = 1}^{r-1}
  ||\na^s v||_{L^2( \D_t)}
  \label{dtplot}
 \end{multline}
\end{lemma}

The elliptic estimates in Section \ref{ellsec}
give us the following coercive estimates.
These are essentially from \cite{Christodoulou2000};
the only difference here is that
these estimates hold when $Vol \D_t = \infty$.
\begin{lemma}
  \label{maincoer}
  Suppose that $\K \leq 1$. Then there are constants $C_r$ with:
 \begin{align}
  ||\na^r v||_{L^2(\D_t)}^2
  &\leq C_r \E_r,\label{coerv}
\end{align}
\begin{align}
  ||\Pi \na^r p||^2_{L^2(\pa \D_t)} &\leq ||\na p||_{L^\infty(\pa \D_t)}
  \E_r.
  \label{pproj}
\end{align}
In addition, for $r \geq 1$:
\begin{equation}
  ||\na^r p||_{L^2(\D_t)}^2 + ||\na^r p||_{L^2(\pa\D_t)}^2
  \leq C_r\big( ||\na p||_{L^\infty(\pa \D_t)} +
  ||\na v||^2_{L^\infty(\D_t)}\big)^2 \E_{r}^*,
  \label{pell}
 \end{equation}
 with $\E_r^* = \sum_{k \leq r} \E_k$, and:
 \begin{multline}
  ||\Pi \na^r D_t p||_{L^2(\pa \D_t)}^2 + ||\na^{r-1} D_tp||_{L^2(\pa \D_t)}^2
  + ||\na^r D_t p||_{L^2(\D_t)}^2\\
   \leq C_r \big(||\na p||_{L^\infty(\D_t)}
  + ||\na v||_{L^\infty(\D_t)}^2 +
  ||\na_n D_t p||_{L^\infty(\pa \D_t)}
  ||\theta||_{L^\infty(\D_t)}\big) \E_r(t)\\
  + P\Big(\E_{r-1}^*, ||\na p||_{L^\infty(\D_t)},
  ||\na v||_{L^\infty(\D_t)}^2, ||\na^2 p||_{L^\infty(\pa \D_t)}\Big).
  \label{dtpests}
 \end{multline}
 Furthermore, if $-\nabla_\n p \geq \delta_0 > 0$, then:
\begin{equation}
 ||\tn^{r-2} \theta||_{L^2(\pa \D_t)}^2
 \leq ||(\na_\n p)^{-1} ||_{L^\infty(\pa \D_t)} \Big( \E_r +
 P(\E_{r-1}^*, ||\na v||_{L^\infty(\D_t)},
 ||\na p||_{L^\infty(\D_t)}, ||\na^2 p||_{L^\infty(\pa \D_t)})\Big)
 \label{coertheta}
\end{equation}
where $P$ is a homogeneous polynomial with positive coefficients.
\end{lemma}
\begin{proof}

The estimate \eqref{coerv} follows from \eqref{pw}
and \eqref{pproj} follows from the defintion of the
boundary term in the energy.
To prove \eqref{pell}, we apply \eqref{mainell}, \eqref{mainlot} and
\eqref{pproj}, which gives \eqref{pell} with an extra term $||\na p||_{L^2(\D_t)}$
on the right-hand side. To control this, we integrate by parts twice and use
\eqref{peq}:
\begin{equation}
 \int_{\D_t} |\na p|^2 = -\int_{\D_t} p \Delta p
 = \int_{\D_t} p \na_i (v^j\na_j v^i)
 = \int_{\D_t} \na_i p (v^j \na_j v^i).
 \label{}
\end{equation}
Bounding the right hand side by $||\na v||_{L^\infty(\D_t)} ||\na p||_{L^2(\D_t)} ||v||_{L^2(\D_t)}$
and dividing both sides by $||\na p||_{L^2(\D_t)}$ gives the result.

Similarly, applying
\eqref{dtplot},
\eqref{mainell} and \eqref{projell1} gives \eqref{dtpests} with
an extra term $||\na D_t p||_{L^2(\D_t)}$ on the right-hand side. This
can be handled by using the fact that $D_t p = 0$ on $\pa \D_t$,
the equation \eqref{dtpdiv} and integrating by parts twice:
\begin{equation}
 \int_{\D_t} |\na D_t p|^2
 = -\int_{\D_t} D_t p (\na_i X^i) = \int_{\D_t} (\na_i D_t p)  X^i,
 \label{dtpxbd}
\end{equation}
where $X_i = v^k \na_k \na_i p - \na_k v_i \na^k p
- D_t (v^j \na_j v_i) + v_i \na_k(v^j \na_j v^k)$. The result now follows after
using
\eqref{pell} and \eqref{coerv} to control $||X||_{L^2(\D_t)}$.

The estimate \eqref{coertheta} follows from \eqref{projell2} and the estimates
we have just proved.
\end{proof}

\subsection{Proof of Theorem \ref{enestthm}}

We start by applying Proposition 5.11 from \cite{Christodoulou2000} with
$\alpha = -\na^r p, \beta = \na^{r-1}v$ and
$\nu = |\na p|^{-1}$, which gives:
\begin{align}
 \frac{d}{dt}\E_r &\leq
 C \sqrt{\E_r} \big( ||\Pi (D_t \na^r p + (\na_k p) \na^r u^k)||_{L^2(\pa \D_t)}
 + ||D_t \na^r u + \na^{r+1} p||_{L^2(\D_t)}\big)\\
 &+ C \K \E_r
 + C \big(||\curl \na^{r-1} v||_{L^2(\D_t)}
 + ||\Delta \na^{r-1} p||_{L^2(\D_t)}\\
 &+ \K( ||\na^{r-1} v||_{L^2(\D_t)}
 + ||\na^r p||_{L^2(\D_t)}\big)^2.
 \label{}
\end{align}

By Lemma \ref{maincoer}, every term except the first one above
is bounded by the right-hand side of \eqref{dten}.
By \eqref{projlot} and \eqref{dtpests}, it suffices to prove the following
bound:
\begin{equation}
 \sum_{s = 1}^{r-2} ||\Pi ( (\na^{1+s} u) \cdot (\na^{r-s} p) )||_{L^2(\pa\D_t)}^2
 \leq  C (K + L + M)\Big( \E_r + (K + L + M) P(\E_0,.., \E_{r-1}, K,L,M),
 \label{projgoal}
\end{equation}
for a polynomial $P$.
We write $(\Pi^{r-s}\na^{r-s}p)_J = \gamma^I_J \na_I^{r-s}$ and
$(\Pi^{s+1}\na^s v)_{J i} = \gamma^I_J \gamma^i_j \na_J v_j$. Then:
\begin{multline}
 ||\Pi \big( (\na^{s+1}v) \cdot (\na^{r-s} p)\big)||_{L^2(\pa \D_t)}
  \leq ||\, |\Pi^{s+1} \na^s v| \, |\Pi^{r-s} \na^{r-s} p| \, ||_{L^2(\pa \D_t)}
 \\+ ||\, |\Pi^s N^k \na^s v_k|\, |\Pi^{r-s} N^k\na^{r-s-1}\na_k p||_{L^2(\pa \D_t)}
 \label{}
\end{multline}
We now apply the interplation
inequality \eqref{interp1} which shows that
see that each of these terms is bounded by a constant depending on $\K$ times
(writing $L^p = L^p(\pa \D_t))$:
\begin{align}
   \bigg( &||\na^2 v||_{L^\infty} + \sum_{\ell = 2}^{r-3} ||\na^\ell v||_{L^2}\bigg)
   ||\nabla^{r-1} p||_{L^2}
   + \bigg( ||\na^3 p||_{L^\infty} + \sum_{\ell = 3}^{r-2} ||\na^\ell p||_{L^2}\bigg)
   ||\nabla^{r-2} v||_{L^2}\\
   +& (1 + ||\theta||_{L^\infty})^{r-4}( ||\theta||_{L^\infty} + ||\tn^{r-3} \theta||_{L^2})
\bigg( ||\na^2 v||_{L^\infty} + \sum_{\ell = 2}^{r-3} ||\na^\ell v||_{L^2}\bigg)
\bigg( ||\na^3 v||_{L^\infty} + \sum_{\ell = 3}^{r-2} ||\na^\ell p||_{L^2}\bigg),
\end{align}
and using Lemma \ref{coerlem}, this can be bounded by the right-hand side of \eqref{projgoal}.

\subsection{Proof of Lemma \ref{A1lem}}

  To control $||\theta||_{L^\infty(\pa \D_t)} + \frac{1}{\iota_0}$
   we start by noting that $\frac{1}{\iota_0} \leq C ||\theta||_{L^\infty(\pa \D_t)}$
  and that by the elementary formula $\theta_{ij} = (1 + |\nabla h|^2)^{-1/2}
  \nabla_i\nabla_j h$, we have $||\theta||_{L^\infty(\pa \D_t)}
   \leq C ||h||_{C^2(\R^2)}$.
  We note that $\Delta |\na^2 \psi|^2 = |\na^3\psi|^2 \geq 0$,
  so writing $v = \na \psi + v_\omega$, applying the maximum
  principle to control $||\na^2 \psi||_{L^\infty(\D_t)} \leq
  ||\na^2\psi||_{L^\infty(\pa \D_t)}$ and
  the estimate \eqref{mainomegabd}, we have:
  \begin{align}
   ||\na v||_{L^\infty(\D_t)} \leq ||\na^2 \psi||_{L^\infty(\D_t)}
   + ||\na v_\omega||_{L^\infty(\D_t)}
   \lesssim ||\na^2 \psi||_{L^\infty(\pa \D_t)} +
   ||\omega||_{H_w^{N_1}(\D_t)}.
   \label{linftyreplace}
  \end{align}
  To control $\na^2 \psi$ on $\pa \D_t$, we can either use \eqref{cr} and
  \eqref{neum} or just use the pointwise inequality \eqref{pw} on
  $\pa \D_t$
  which shows that $|\na^2 \psi| \lesssim |\Delta \psi| + |\Pi\na^2 \psi|$.
   By the projection
  formula \eqref{basicproj} we have $|\Pi \na^2 \psi| \leq |\tn^2 \psi| + |\theta|
  (|\na_N \psi| + |\tn \psi|) \lesssim |\tn^2\varphi| + |\theta|
  (|\N \varphi| + |\tn \varphi|)$ where $\tn$ denotes the covariant derivative
  on $\pa \D_t$.
  By the estimate for the Dirichlet-to-Neumann map \eqref{dnmap}, this
  proves the bound for $||\na v||_{L^\infty(\pa \D_t)}$.

  The estimates for $||\na^2 p||_{L^\infty(\pa \D_t)}$ follow
  from the pointwise estimate \eqref{pw}, the fact that
  $\Delta p = -(\na v)\cdot (\na v)$ and the bounds we just proved.
  To bound $||\na D_t p||_{L^\infty(\pa \D_t)}$, we apply Sobolev
  embedding \eqref{appsobbdy} on $\pa \D_t$ and the elliptic estimate \eqref{mainell2}
  it suffices to bound:
  \begin{equation}
   ||\Pi \na^3 D_t p||_{L^2(\pa \D_t)}
   + ||\Pi \na^2 D_t p||_{L^2(\pa \D_t)}
   + \sum_{s \leq 2} ||\na^s \Delta D_t p||_{L^2(\D_t)}
   +|| \na D_t p||_{L^2(\D_t)}.
 \end{equation}
  Using the identity \eqref{dtpxbd} gives:
  \begin{equation}
   ||\na D_t p||_{L^2(\D_t)}
   \leq C \Big( ||\na^2 p||_{L^\infty(\D_t)} ||v||_{L^2(\D_t)}
   + ||\na v||_{L^\infty(\D_t)}||\na p||_{L^2(\D_t)}
   + ||\na v||_{L^\infty(\D_t)}^2 ||\na v||_{L^2(\D_t)}\Big),
   \label{}
  \end{equation}
  and using the estimates we have just proved and Lemma \ref{coerlem}
  gives that $||\na D_t p||_{L^2(\D_t)}$
  is bounded by the right-hand side of \eqref{a1est}. To control
  $||\Pi \na^3 D_t p||_{L^2(\pa \D_t)} +
  ||\Pi \na^2 D_t p||_{L^2(\pa \D_t)}$, we use the formulas
  \eqref{basicproj},
  \eqref{3rdorder} and the estimates we have just proved.

  To get a lower bound for $\nabla_N p$ on $\pa \D_t$, we
  start by noting that since $p = 0$ on $\pa \D_t$ and
  $(D_t N^i) N_i = 0$, so that $D_t \na_N p = \na_N D_t p$ on $\pa \D_t$.
Since $p = 0$ on $\pa \D_t$ and $(D_t N^i)N_i = 0$ it follows that
$D_t \pa_N p = (D_t N^i) \pa_i p + \pa_N D_tp = \pa_N D_t p$. Applying
Sobolev embedding on $\pa \D_t$, the estimate \eqref{dtpests}, and
the bootstrap assumptions \eqref{bootstrap1}-\eqref{bootstrap3}, we have:
\begin{equation}
 |\nabla_N p(t)|  \geq |\nabla_N p(0)| -
 \int_0^t |\nabla_N D_t p(s)|\, ds
 \geq |\nabla_N p(0)|
 - C \int_0^t \frac{\ve_0^3}{(1+ s)^2}(1+s)^\delta
 +\ve_1^3 (1+s)^{3\delta}\, ds.
 \label{}
\end{equation}
The second term is bounded by $\frac{1}{2}|\nabla_N p(0)|$ so long as
$t \leq C(|\nabla_N p(0)|^{-1}) \ve_1^{-1/3}$ and $\ve_0$ is
taken sufficiently small.

\subsection{Proof of Proposition \ref{uenests}}
\label{coerensec}

We now show how the energies in the previous section control Sobolev
norms of $\varphi, h$. Recall that $u =  h + i \Lambda^{1/2}
\varphi_\omega$, where $\varphi_\omega = \varphi + a_\omega$ and $\nabla a_\omega = V_\omega
+ \nabla h B_\omega$.

We begin by noting that
by the definition of $\varphi_\omega = \varphi + a_\omega$,
 the fact that $\nabla a_\omega =
V_\omega + \nabla h B_\omega$, and the fact that $B_\omega = -\nabla h\cdot V_\omega$
(since $v_\omega \cdot N = 0$)
it suffices to prove the following estimate:
\begin{align}
  ||h||_{H^r(\R^2)}^2 +
||\Lambda^{1/2} \varphi||_{L^2(\R^2)}^2 + ||\nabla \varphi||_{H^{r-1}(\R^2)}^2
+ ||\Lambda^{1/2} a_\omega||_{L^2(\R^2)}^2
+ ||V_\omega||_{H^{N-1}(\R^2)}^2
\lesssim \E^N + \A P(\E^{N-1}).
 \label{}
\end{align}

We start with bounds for $h$.
By the elementary formula:
\begin{align}
 \theta_{ij} = \frac{1}{\sqrt{1 + |\nabla h|^2}} \nabla_i \nabla_j h
 \label{thetah}
\end{align}
we have $\nabla^r h \sim \nabla^{r-2}\theta + O(\nabla^{r-1}h,...,
\nabla h)$. We can therefore bound $||h||_{H^{N}(\R^2)}$ by the right-hand
side of \eqref{rhs1} provided we also control
$||\nabla h||_{L^2(\R^2)} + ||h||_{L^2(\R^2)}$.
Note that $||h||_{L^2(\R^2)} \leq E_0$ where $E_0$ is the
conserved energy (defined in \eqref{conserved}), and a bound for $||\nabla h||_{L^2(\R^2)}$
follows from this and the bound for $||\nabla^2 h||_{L^2(\R^2)}$.

We now bound $\varphi$.
First, we have:
\begin{equation}
 -\int_{\pa \D_t} \varphi \N \varphi = -\int_{\pa \D_t}
 \psi \na_N \psi = -\int_{\D_t} \psi \Delta \psi +
 \int_{\D_t} |\nabla \psi|^2 \leq ||v||_{L^2(\D_t)}^2 + ||v_\omega||_{L^2(\D_t)}^2.
 \label{}
\end{equation}
The left-hand side is:
\begin{equation}
 ||\N^{1/2} \varphi||_{L^2(\R^2)}^2 \sim ||\Lambda^{1/2} \varphi||_{L^2(\R^2)}^2,
 \label{}
\end{equation}
which follows from the remarks after Proposition 2.2 in \cite{Germain2012}.

To control $\Lambda^{1/2} a_\omega$, we note that by the fractional
integration estimate \eqref{fracint}, $||\Lambda^{1/2} a_\omega||_{L^2(\R^2)}
= ||\Lambda^{-1/2}\Lambda a_\omega||_{L^2(\R^2)}
\lesssim ||\Lambda a_\omega||_{L^{4/3}(\R^2)}$ and by the
fact that the Riesz transform is bounded on $L^{4/3}$ it follows that
$||\Lambda^{1/2} a_\omega||_{L^2(\R^2)} \lesssim ||\nabla a_\omega||_{L^{4/3}(\R^2)}$.
Since $\nabla a_\omega = V_\omega + \nabla h B_\omega$, we have
\begin{equation}
 ||\Lambda^{1/2} a_\omega||_{L^2(\R^2)}
 \lesssim
 ||(1 + |x|^2)^{1 /2} V_\omega||_{L^{2}(\R^2)}
 + ||\nabla h||_{L^\infty(\R^2)} ||(1 + |x|^2)^{1/2} B_\omega||_{L^2(\R^2)}
 \label{}
\end{equation}
and by \eqref{mainomegabd}, this is bounded by the right-hand
side of \eqref{rhs1}.

To control the higher norms of $\varphi$ and $V_\omega$, we
use the following:
\begin{lemma}
  Under the hypotheses of Proposition \ref{uenests}, we have:
  \begin{equation}
   ||\nabla^r \varphi||_{L^2(\R^2)}^2 + ||\nabla^{r-1} V_\omega||_{L^2(\R^2)}^2
   \lesssim \E^r + \A P(\E_*^{r-1}),
   \label{bdynormcontrol}
  \end{equation}
  where $\E_*^{r-1} = \sum_{s \leq r-1} \E^{r-1}$ and $\A$
  is defined by \eqref{Adef}.
\end{lemma}
\begin{proof}
The estimates for $V_\omega$ follow from \eqref{mainomegabd}.
To bound $\nabla^r \varphi$,
  we start with the fact that:
  \begin{equation}
   ||\na \psi||_{L^2(\D_t)} \lesssim ||v||_{L^2(\D_t)} + ||v_\omega||_{L^2(\D_t)},
   \label{napsibd}
  \end{equation}
By the chain rule, we have:
\begin{equation}
 ||\nabla \varphi||_{L^2(\R^2)} \leq||\na \psi||_{L^2(\pa\D_t)}
 + ||\nabla h \na_y\psi||_{L^2(\pa\D_t)}.
 \label{}
\end{equation}
Bounds for the second term will follow in a similar way to the
bounds for the first term so we just show how to bound the first term.
By the inequality \eqref{traceproj}:
\begin{align}
 ||\na \psi||_{L^2(\pa\D_t)}^2
 &\leq
 ||\na_N \psi||_{L^2(\D_t)}^2 + ||\Delta \psi||_{L^2(\D_t)}^2
 + K||\na \psi||_{L^2(\D_t)}^2\\
 &\leq
 ||v\cdot N||_{L^2(\pa\D_t)} + K||\na \psi||_{L^2(\D_t)},
\end{align}
and so the trace inequality \eqref{trace} and the estimate \eqref{napsibd}
imply:
\begin{equation}
 ||\na \varphi||_{L^2(\R^2)}
 \lesssim ||\na v||_{L^2(\D_t)} + ||v||_{L^2(\D_t)} + ||v_\omega||_{L^2(\D_t)},
 \label{}
\end{equation}
where the implicit constant depends only on $K$. The first two terms
are bounded by $\E^1 + \E^0$ and the last term can be bounded by
$||\omega||_{H^{N_1}_w(\D_t)}$ by \eqref{vomegabd}.
To explain the strategy for higher-order derivatives we first consider
what happens when $r = 2$. Using \eqref{traceproj} again:
\begin{equation}
 ||\na^2 \psi||_{L^2(\pa \D_t)}
 \lesssim ||\na_N \na \psi||_{L^2(\pa \D_t)}
 + K||\na \psi||_{L^2(\D_t)}.
 \label{}
\end{equation}
By the estimate \eqref{traceproj}:
\begin{equation}
 ||\na_N \na \psi||_{L^2(\pa \D_t)}
 \lesssim || \Pi \na_N \na \psi||_{L^2(\pa \D_t)}
 + ||\div \na_N \na \psi||_{L^2(\D_t)}
 + ||\curl \na_N \na \psi||_{L^2(\D_t)}
 + K||\na_N \psi||_{L^2(\D_t)}.
 \label{}
\end{equation}

Note that:
\begin{equation}
 \Pi^i_j \na_N \na_i \psi
 = \Pi^i_j \na_i \na_N \psi
 - (\Pi^i_j\na_i N^k) \na_k \psi.
 \label{}
\end{equation}
The first term is $\tn (v\cdot N)$ and the second term is
$-\theta_i^k\na_k \psi$.
Also both $\div \na_N \na \psi$ and $\curl \na_N \psi$ are lower order.
The first is because to highest order it is $\na_N \Delta \psi = 0$ and the
second because $\curl \na \psi = 0$.

 Therefore we have:
\begin{equation}
 ||\na^2 \psi||_{L^2(\pa \D_t)}
 \lesssim ||\tn(v\cdot N)||_{L^2(\pa \D_t)}
 + K ||\na \psi||_{L^2(\pa \D_t)}
 + ||\na \psi||_{L^2(\D_t)}.
 \label{}
\end{equation}
Using the trace inequality to bound the first term and the above argument
to bound the lower-order norms of $\psi$ gives that:
\begin{equation}
 ||\na^2 \psi||_{L^2(\pa \D_t)}
 \lesssim ||\na^2 v||_{L^2(\D_t)} ||\na v||_{L^2(\D_t)} +
 ||v||_{L^2(\D_t)} + ||v_\omega||_{L^2(\D_t)},
 \label{}
\end{equation}
where the implicit constant depends only on $K$.

We now prove a higher-order version of this.
Repeatedly applying the chain rule \eqref{cr}, to highest order we have:
\begin{equation}
 \nabla^r \varphi \sim \nabla^r \psi + \nabla^r h (\na_y \psi)
 + ...
 \label{}
\end{equation}
where the missing terms are all bounded pointwise by
$\sum_{k \leq r-1} |\na_{x,y}^k \psi|$ times a
polynomial in
$\sum_{k \leq r-1} |\nabla^k h|$.
We now want to replace $\nabla^r\psi$ with $\nabla^{r-1} \nabla_N \psi
\sim \tn^{r-1} (v\cdot N)$ and lower order terms. By the inequality
\eqref{traceproj}:
\begin{equation}
 ||\nabla^r \psi||_{L^2(\pa \D_t)}
 \lesssim
 ||\nabla_N \nabla^{r-1} \psi||_{L^2(\pa \D_t)}
 + ||\nabla^{r-1} \psi||_{L^2(\D_t)},
 \label{}
\end{equation}
with implicit constant depending on $K$. Next, with $\beta = \nabla^{r-1}\psi$,
we apply the estimate \eqref{traceproj} and have:
\begin{multline}
 ||\nabla_\n \nabla^{r-1} \psi||_{L^2(\pa \D_t)}\\
 \lesssim ||\Pi \nabla_\n \nabla^{r-1} \psi||_{L^2(\pa \D_t)}
 + ||\div \nabla_{\n} \nabla^{r-2}\psi||_{L^2(\D_t)} +
 ||\curl \nabla_{\n} \nabla^{r-2}\psi||_{L^2(\D_t)}
 + ||\nabla_{\n} \nabla^{r-2}\psi||_{L^2(\D_t)}.
 \label{}
\end{multline}
The interior terms are all lower order by the same observation as above,
and so we just need to deal with the boundary term.
We note that:

\begin{equation}
 \Pi^I_J \nabla_\n \nabla_I^{r-1} \psi
 = \Pi^{I}_J \nabla_I^{r-1} \nabla_\n \psi
 - \sum_{K,L} \Pi^I_J (\nabla_K \n^k) (\nabla_L^{r-s} \nabla \psi)
 \label{anexpansion}
\end{equation}
where the sum is over all multi-indices $K, L$ with
$K+L = I$ and $|K| \leq |I|-1$.

Since $\nabla_\n \psi = \n\cdot v$ on $\pa \D_t$, using
\eqref{projlem} to replace $\Pi^I_J \nabla_I^{r-1} \nabla_\n \psi$
with $\overline{\na}^{r-1} \nabla_\n \psi$ and applying
Lemma \ref{coerlem} to control $\overline{\na}^{r-1} (\n \cdot v)$ by
the energy shows that the first term in \eqref{anexpansion} is controlled
by the energy. The worst term appearing in the sum in \eqref{anexpansion}
from the point of view of the regularity of $\theta$ is
the case $K = I$. This involves $r-1$ projected derivatives of $\n$
and by Proposition 4.11 of \cite{Christodoulou2000}
and the definition $\theta = \Pi \nabla N$, this can
be bounded by $||\overline{\na}^{r-2} \theta||_{L^2(\pa \D_t)}$ to
highest order. We can now use induction and interpolation
\eqref{interp1} to deal with the lower-order terms.

Having now bounded $\varphi$, let us see how to control $V_\omega$
and $B_\omega$. First, since $v_\omega \cdot \n = 0$ on $\pa \D_t$,
we have $B_\omega = V_\omega \cdot \nabla h$ and so it is enough
to bound $V_\omega$. Since $V_\omega = v_\omega|_{\pa \D_t}
= (v - \nabla \psi)|_{\pa \D_t}$, estimates for $V_\omega$ follow
from the above estimates for $\psi$ and the estimates in Lemma
\ref{maincoer}.

\end{proof}

\subsection{Proof of Proposition \ref{uenlem}}
\label{vortensec}
A short calculation using the fact that $[D_t, \na] = - \na v^k\na_k$,
$D_t (1 + |z|^2)^{2} = 4 |z|^2 z\cdot v$  and the
equation for the vorticity \eqref{omegaeq} shows that:
\begin{equation}
 D_t \na^m ((1 + |z|^2)^{2} \omega)
 = (1 + |z|^2)^{2}\bigg( \na^{m+1} v \cdot \omega
 + \na v \cdot \na^m \omega\bigg) + R,
 \label{}
\end{equation}
where $R$ is a sum of terms which can be bounded pointwise by
$(1 + |z|^2)^{2}\sum_{a = 1}^{m} | \na^{a}v(z)| |\na^{m-a} \omega(z)|$.
We next write $v = \na \psi + v_\omega$ and the result as:
\begin{align}
 D_t ( (1 + |z|^2)^{2} \na^m\omega)
 = (1 + |z|^2)^{2} \bigg( \na^{m+2}\psi \cdot \omega
 + \na^2 \psi \cdot \na^m \omega
 + \na^{m+1} v_\omega \cdot \omega
 + \na v_\omega \cdot \na^m \omega\bigg)
 + R.
 \label{}
\end{align}
Taking $m \leq N_0$
By the Reynolds transport theorem, the above calculation
and Sobolev embedding, we have:
 \begin{align}
  \frac{d}{dt} &||\na^m \omega(t)||_{L^2_w}^2\\
   &\lesssim
   \int_{\D_t} (1 + |z|^2)^{2} \big(|\na^{m+2}\psi|  |\omega|
   + |\na^2 \psi| |D^m \omega|
   + |\na^{m+1}v_\omega| |\omega| + |\na v_\omega| |\na^m\omega|
   + R \big)
   |\na^m \omega| \, dz\\
   &\lesssim\Big(
   ||\na^2 \psi||_{W^{m,\infty}(\D_t)}
   + ||\na^{m+1}v_\omega||_{L^2(\D_t)}
   +
   ||v_\omega||_{L^\infty(\D_t)} \Big)||\omega||_{H^{N_1}_w(\D_t)}^2
  \label{rhsvort}
\end{align}
To control the first term, we use the maximum principle as in the
proof of Lemma \ref{A1lem}, which gives that
$||\na^2 \psi||_{W^{m,\infty}(\D_t)}
\leq ||\na^{2}\psi||_{W^{m,\infty}(\pa \D_t)}$.
Using \eqref{cr} and
\eqref{neum} repeatedly shows that $||\na^s \psi||_{L^\infty(\pa \D_t)}
\lesssim || \na \varphi||_{W^{s-1,\infty}(\R^2)} + ||h||_{W^{s,\infty}(\R^2)}
\lesssim ||u||_{W^{s+1,\infty}(\R^2)}$, up to lower order terms.

To control the other two terms from \eqref{rhsvort}, we use
\eqref{mainomegabd}:
\begin{align}
  ||v_\omega||_{L^\infty(\D_t)}
  + ||\na^{m+1} v_\omega||_{L^2(\D_t)} &\lesssim ||\omega||_{H^{N_1}_w(\D_t)},
 \label{}
\end{align}
which proves \eqref{wenest}. \qedhere

We also note the following, which is used in the proof of Corollary \ref{mainthm2}:
\begin{lemma}
 If $\omega_0|_{\pa \D_0} = 0$ and $\int_0^T ||\pa v||_{L^\infty(\pa \D_s)} < \infty$,
 for some $T > 0$, then $\omega|_{\pa \D_t} = 0$ for $t \leq T$.
\end{lemma}
\begin{proof}
  Changing to Lagrangian coordinates and letting $\mu_\gamma$
  denote the volume element on $\pa \Omega$ with respect to the
  metric $\gamma$, we have:
 \begin{equation}
  \frac{d}{dt} \int_{\pa \D_t} |\omega(t)|^2\, dS
  = \frac{d}{dt} \int_\Omega |\omega(t)|^2 d \mu_\gamma
  = \int_{\Omega} D_t \omega(t) \cdot \omega (t)
  + |\omega(t)|^2 D_t d\mu_\gamma.
  \label{dtomegabd}
 \end{equation}
 By Lemma 3.9 in \cite{Christodoulou2000}, we have
 $D_t d \mu_\gamma = (\tr h = h_{\n\n}) d\mu_\gamma$ where
 $h = \frac{1}{2} D_t g$ with $g$ the metric in Lagrangian
 coordinates (defined in \eqref{metric}) and $h_{\n\n} =
 h(\n, \n)|$. A simple calculation using
 \eqref{metric} and the fact that
 $D_t \frac{d}{dy} x^i = \frac{d}{dy} V^i$
 gives that $|D_t d \mu_\gamma| \leq C||\pa v||_{L^\infty(\pa \Omega)} d\mu_\gamma$,
 so by \eqref{omegaeq}, \eqref{dtomegabd} gives:
 \begin{equation}
  \frac{d}{dt} ||\omega(t)||_{L^2(\pa \D_t)}^2
  \leq C ||\pa v||_{L^\infty(\pa \D_t)} ||\omega(t)||_{L^2(\pa \D_t)}^2.
  \label{}
 \end{equation}
 Multiplying both sides by the integrating factor $e^{-C \int_0^t||\pa v(s)||_{L^\infty(\pa \D_s)}\, ds}$
 and integrating gives that:
 \begin{equation}
  ||\omega(t)||_{L^2(\pa \D_t)}^2
  \leq C \exp\Big( \int_0^t ||\pa v(s)||_{L^\infty(\pa \D_s)}\Big)
  ||\omega(0)||_{L^2\pa \D_0)}^2,
  \label{}
 \end{equation}
 from which the result follows.
\end{proof}

\section{Dispersive estimates for terms involving the vorticity}
\label{vortests}

We now prove the estimates for the terms $g_2,..., g_5$
from \eqref{duha2}. We recall
that $R_j$ denotes the Riesz transform and $\Lambda^s$ denotes fractional
differentiation on $\R^2$. We will also ignore the difference between $Re u,
Im u$ and just write $u$.
Then the terms we want to estimate are:
\begin{align}
 g_2(t) &= \int_{0}^t e^{is\Lambda^{1/2}}  R\cdot V_\omega(s)\,ds,\\
 g_3(t) &= \int_0^t e^{is\Lambda^{1/2}} \Lambda^{1/2}
 \big( (R\cdot  V_\omega)(\Lambda^{1/2} u)\big)
 - \nabla \cdot \big(  u  V_\omega\big)
 + \Lambda( u R\cdot V_\omega),\\
 g_4(t) &= \int_0^t e^{is\Lambda^{1/2}} \Lambda^{1/2}
 (R\cdot V_\omega)^2\, ds.
 \label{}
\end{align}

In the next three sections, we prove:
\begin{prop}
  If \eqref{bootstrap1} holds with $\ve_0 \ll 1$, then:
  \begin{multline}
   ||\nabla^k e^{-it\Lambda^{1/2}} g_I(t)||_{L^\infty(\R^2)}
   \lesssim \int_0^t \Big( 1 + ||\omega(s)||_{H^{N_1}_w(\D_s)} +
   ||u(s)||_{W^{k+3,\infty}(\R^2)} \Big)
   ||\omega(s)||_{H^{N_1}_w(\D_s)}
   \, ds,
   \label{inftyvort}
  \end{multline}
  for $k \leq N_1 + 4$,
  and
 \begin{multline}
  ||\Lambda^{\iota} x g_I(t)||_{L^2(\R^2)}\\
  \lesssim \int_0^t (1+s) \bigg(1 + ||\omega(s)||_{H^{N_1}_w(\D_s)}
  + ||u(s)||_{W^{4, \infty}(\R^2)} + ||u(s)||_{H^{N_0}(\R^2)}\bigg)
  ||\omega(s)||_{H^{N_1}_w(\D_s)}\, ds,
  \label{weightedvort}
 \end{multline}
 for $I = 2,3,4$.
\end{prop}

Assuming this holds for the moment, we show how it implies the estimates
for $f_2$ in \eqref{fbdsA}-\eqref{fbdsB}. If the assumptions \eqref{bootstrap1}-
\eqref{bootstrap3} hold, then using Lemma \ref{interpolationgms},
\eqref{inftyvort} implies:
\begin{equation}
 ||\nabla^k e^{-it\Lambda^{1/2}} g_I(t)||_{L^\infty(\R^2)}
 \lesssim (1 + t)^{1+\delta}\ve_1 +
 (1+t)^{-1+\sigma} \ve_0 \ve_1 +  (1+t)^{2\delta}\ve_1^2
 \label{}
\end{equation}
where $\sigma \leq \frac{1}{N}$. Since $\ve_1 \ll \ve_0$, this implies
the second inequality in \eqref{fbdsA}.

Similarly, we have:
\begin{equation}
 ||\Lambda^\iota (x g_I(t))||_{L^2(\R^2)} \lesssim
 (1 +t)^{2+\delta} \ve_1
 + (1+t)^{2+2\delta} \ve_0 \ve_1
 + (1+t)^{2+2\delta} \ve_1^2,
 \label{}
\end{equation}
which implies the second inequality in \eqref{fbdsB}.

\subsection{Estimates for $g_2$}
\label{vortest1}

\begin{lemma}
  \label{g2estlem}
 If $v$ satisfies \eqref{bootstrap1},\eqref{bootstrap2} with
 $\ve_0 \ll 1$, then:
 \begin{align}
  ||\nabla^k e^{-it\Lambda^{1/2}} g_2(t)||_{L^\infty(\R^2)}
  &\lesssim
  \int_0^t ||\omega(s)||_{H^{N_1}_w(\D_s)}\, ds,
  && k \leq N_1 - 2\label{g2dec}\\
  ||\Lambda^\iota x g_2(t)||_{L^2(\R^2)}
  &\lesssim \int_0^t (1 + s) ||\omega(s)||_{H^{N_1}_w(\D_s)}\, ds.
  \label{g2wt}
 \end{align}
\end{lemma}

\begin{proof}
By Sobolev embedding, we have:
\begin{equation}
 ||\nabla^k e^{i(t-s)\Lambda^{1/2}} \Lambda^{\iota} R \cdot V_\omega(s)||_{L^\infty(\R^2)}
 \lesssim ||R \cdot V_\omega(s)||_{H^{3/2 + k}(\R^2)}
 \lesssim ||V_\omega(s)||_{H^{2 + k}(\R^2)}.
 \label{}
\end{equation}
By  \eqref{vomegabd},  $||V_\omega(s)||_{H^{2+k}(\R^2)} \lesssim
||\omega(t)||_{H^{N_1}_w(\D_s)}$, which implies \eqref{g2dec}. Note
that this estimate loses more than half a derivative, but we are avoiding
the use of fractional derivatives in $\D_t$.

By Plancherel's theorem, $||\Lambda^\iota x e^{it\Lambda^{1/2}}g_2||_{L^2}
= |||\xi|^\iota \pa_\xi (e^{it|\xi|^{1/2}} \hat{g}_2)||_{L^2}$.
We have:
\begin{align}
 \pa_\xi e^{it|\xi|^{1/2}}\hat{g}_2 &=
 \pa_\xi \bigg( \int_0^t  e^{is|\xi|^{1/2}}
 \frac{\xi}{|\xi|}\cdot \hat{V}_\omega(s,\xi)\, ds\bigg)\\
 &= \int_0^t s \frac{\xi}{|\xi|^{3/2}}
 e^{is |\xi|^{1/2}}
 \hat{V}_\omega(s,\xi)\, ds
 + \int_0^t
 e^{is |\xi|^{1/2}} \pa_\xi \bigg(\frac{\xi}{|\xi|}
 \bigg)
 \hat{V}_\omega(s,\xi)\, ds
 + \int_0^t
 e^{is |\xi|^{1/2}}
 \pa_\xi \hat{V}_\omega(s,\xi)\, ds
 \\
 &\equiv \int_0^t s\hat{g}_2^1(s,\xi)\, ds +
 \int_0^t \hat{g}_2^2(s,\xi) \,ds
  + \int_0^t \hat{g}_2^3(s,\xi)\, ds.
 \label{}
\end{align}

By the fractional integration lemma \eqref{fracint}, we have:
\begin{align}
 ||\Lambda^\iota g_2^1||_{L^2}
 \lesssim ||\Lambda^{-1/2+\iota} V_\omega||_{L^2}
 \lesssim ||V_\omega||_{L^{p_1}},
 \label{}
\end{align}
where $p_1 = 2(2-\iota)/3$.

Similarly, bounding $||\xi|^{\iota} \pa_\xi (|\xi|^{-1} \xi)|
\lesssim |\xi|^{-1+\iota}$ and taking $p_2 = 2/(2-\iota) > 1$,
we have that:
\begin{align}
 ||\Lambda^\iota g_2^2||_{L^2}
 \lesssim \int_0^t ||\Lambda^{-1+\iota} V_\omega(s)||_{L^2}
 \lesssim \int_0^t ||V_\omega(s)||_{L^{p_2}}\, ds
 \lesssim \int_0^t ||\omega(s)||_{H^{N_1}_w(\D_s)}\,ds,
 \label{bad}
\end{align}
by \eqref{lowp}.

To control $\Lambda^\iota g_2^3$,
we write $\Lambda^\iota = \Lambda^{\iota-1} \Lambda = -\Lambda^{\iota-1}
R\cdot \nabla$. Using fractional integration again, we have:
\begin{align}
 ||\Lambda^\iota(x V_\omega)||_{L^2} \lesssim
 ||\Lambda^{-1+\iota} \nabla(x V_\omega)||_{L^2}
 \lesssim ||\nabla(x V_\omega)||_{L^{p_2}},
 \label{}
\end{align}
Combining the above estimates and using Proposition \ref{vomegabd} gives
\eqref{g2wt}.
\end{proof}

\subsection{Estimates for $g_3$}
\label{vortest2}
We now bound the term involving both $u$ and $v_\omega$. This is:
\begin{align}
 g_3(t) = \int_0^t e^{is\Lambda^{1/2}} N_1 (u, w)\, ds
 = \sum_{I = 1,2,3} \int_0^t e^{is\Lambda^{1/2}}
 g_3^I(s)\, ds,
 \label{}
\end{align}
with:
\begin{align}
 g_3^1 &= \Lambda^{1/2} \big( (R\cdot V_\omega)
 (\Lambda^{1/2} u)\big),\\
 g_3^2 &= -\nabla\cdot (u V_\omega),\\
 g_3^3 &=
 \Lambda( u R\cdot V_\omega)
 \label{}
\end{align}

\begin{lemma}
  \label{g3estlem}
 If \eqref{bootstrap1}-\eqref{bootstrap3} hold with $\ve_0 \ll 1$, then:
 \begin{align}
  ||\nabla^k e^{-it\Lambda^{1/2}} g_3(t)||_{L^\infty(\R^2)}
  &\lesssim \int_0^t ||\omega(s)||_{H^{N_1}_w(\D_s)}
  ||u(s)||_{W^{k+3,\infty}(\R^2)}\,
  ds\label{g3infty}\\
  ||\Lambda^\iota (x g_3(t))||_{L^2} &\lesssim \int_0^t (1+s) ||\omega(s)||_{H_w^{N_1}(\D_s)}
  \Big(||u(s)||_{W^{4,\infty}} + ||u(s)||_{H^{N_0}}
  \Big)\, ds
  \label{g3wt}
 \end{align}
\end{lemma}

\begin{proof}
  The estimates for each of the terms $g_3^I, I = 1,2,3$ are similar, so we just show how
  to bound $g_3^1$.
  Applying Sobolev embedding and using the fact that the Riesz transform
  maps $L^2 \to L^2$, we have:
  \begin{equation}
   ||\nabla^k \Lambda^{1/2} (R\cdot V_\omega \Lambda^{1/2} u)||_{L^\infty}
   \lesssim ||(R\cdot V_\omega) (\Lambda^{1/2} u)||_{H^{k+2}}
   \lesssim ||V_\omega||_{H^{k+2}} ||u||_{W^{k+3,\infty}},
   \label{}
  \end{equation}
  say. By \eqref{vomegabd} this is bounded by the right-hand side of \eqref{g3infty}.

  To prove the bound for $x g_3$, we write:
  \begin{align}
   \pa_\xi \hat{g}_3^1
   &= \int_{\R^2} \pa_\xi \bigg(e^{is|\xi|^{1/2}}
   |\xi|^{1/2 }  \frac{(\xi-\eta)_\ell}{|\xi-\eta|}
   |\eta|^{1/2} \hat{V}^\ell_\omega(\xi-\eta) \hat{u}(\eta)\bigg)\, d\eta\\
   &= \int_{\R^2}
   e^{is|\xi|^{1/2}}|\eta|^{1/2}
   \bigg( is \frac{(\xi-\eta)_\ell}{|\xi-\eta|}
   + |\xi|^{1/2} m_0(\xi-\eta)
   + m_1(\xi)\frac{(\xi-\eta)_\ell}{|\xi-\eta|} \bigg)
   \hat{V}^\ell_\omega(\xi-\eta) \hat{u}(\eta)\bigg)\,d\eta
   \label{symbolderivs}
   \\
   &+ \int_{\R^2}
   e^{is|\xi|^{1/2}}
   |\xi|^{1/2} \frac{(\xi-\eta)_\ell}{|\xi-\eta|}|\eta|^{1/2}
  \pa_\xi \hat{V}^\ell_\omega(\xi-\eta)
   \hat{u}(\eta)\, d\eta,
   \label{weightterm}
  \end{align}
  with $m_0(\xi,\eta) = \pa_\xi \bigg( \frac{\xi-\eta}{|\xi-\eta|}\bigg)$
  and $m_1(\xi) = \pa_\xi |\xi|^{1/2}$.

  In physical space, after applying $\Lambda^\iota$
  the first term in \eqref{symbolderivs} is:
  \begin{equation}
   is e^{is\Lambda^{1/2}}  \Lambda^\iota \bigg ( (R \cdot V_\omega)
   (\Lambda^{1/2}u)\bigg),
   \label{}
  \end{equation}
  and using the fractional product rule \eqref{fracprod}, this
  is bounded by the right-hand side of \eqref{g3wt}.

  The second term in \eqref{symbolderivs} contributes:
  \begin{align}
    e^{is\Lambda^{1/2}} \Lambda^{1/2+\iota} \bigg(
    (m_0(\nabla) V_\omega) (\Lambda^{1/2} u)\bigg).
   \label{}
  \end{align}
  Since $|m_0(\xi-\eta)| \lesssim |\xi-\eta|$, we bound the result in
  $L^2$ by:

  \begin{equation}
   ||\Lambda^{-1} V_\omega||_{L^{p_1}} ||\Lambda^{\iota+1} e^{is\Lambda^{1/2}} u||_{L^{p_2}}
   + || \Lambda^{-1/2+\iota} V_\omega||_{L^4}
   ||\Lambda^{1/2}e^{is\Lambda^{1/2}} u||_{L^{4}}
   \label{bs2}
  \end{equation}
  where $1/p_1 + 1/p_2 = 1/2$. We then take $p_1$ so large that
  $||\Lambda^{-1} V_\omega||_{L^{p_1}}
  \lesssim ||V_\omega||_{L^{4/3}}$ (see \eqref{fracint}), say, and since $p_2 > 2$
  we bound $||\Lambda^{\iota+1} u||_{L^{p_2}} \lesssim
  ||u||_{H^{N_0}} + ||u||_{W^{4,\infty}}$. Again using
  \eqref{fracint}, the second term can be
  bounded by $||V_\omega||_{L^2} (||u||_{W^{4,\infty}} + ||u||_{H^{N_0}})$.
  Using \eqref{vomegabd} to control the factors of $V_\omega$, the result
  is bounded by the right-hand side of \eqref{g3wt}.

  The third term in \eqref{symbolderivs} is:
  \begin{align}
   e^{is\Lambda^{1/2}}
   m_1(\nabla) \bigg(  (R \cdot V_\omega) (\Lambda^{1/2} u)\bigg),
   \label{}
  \end{align}
  and recall $|m_1(\xi)| \lesssim |\xi|^{-1/2}$. With $1 < p < 2$ so that
  $||\Lambda^{-1/2+\iota} F||_{L^2} \lesssim ||F||_{L^p}$, we therefore have:
  \begin{equation}
   ||\Lambda^{-1/2 + \iota} \Big( (R\cdot V_\omega)(\Lambda^{1/2} u) \Big)||_{L^2}
   \lesssim ||(R\cdot V_\omega )(\Lambda^{1/2} u)||_{L^p}
   \lesssim ||V_\omega||_{L^p} ||u||_{W^{4,\infty}}.
   \label{}
  \end{equation}
The estimate for \eqref{weightterm} can be performed similarly to
how we controlled $g_2^3$ in the previous section.
\end{proof}

\subsection{Estimates for $g_4$}
\label{g4ests}
We now bound the term which is quadratic in the vorticity:
\begin{equation}
 g_4(t) = \int_0^t e^{is\Lambda^{1/2}} N(w,w)\, ds
 = \int_0^t e^{is\Lambda^{1/2}} \Lambda^{1/2 + \iota} (R\cdot V_\omega(s))^2\,ds
 \label{}
\end{equation}

We prove:
\begin{lemma}
 If $v$ satisfies \eqref{bootstrap1}-\eqref{bootstrap3} with $\ve_0 \ll 1$,
 then:
 \begin{align}
  ||\nabla^k e^{-it\Lambda^{1/2}} g_4(t)||_{L^\infty(\R^2)}
  &\lesssim \int_0^t ||\omega(s)||_{H_w^{N_1}(\D_s)}^2\, ds,\label{g4est1}\\
  ||\Lambda^\iota x g_4(t)||_{L^2(\R^2)}
  &\lesssim
  \int_0^t (1 + s) ||\omega(s)||_{H^{N_1}_w(\D_s)}^2\, ds
  \label{}
 \end{align}
\end{lemma}

\begin{proof}
  The argument is nearly identical to the proof of the estimates for $g_3$.
  We start with:
  \begin{equation}
   ||\nabla^k \Lambda^{1/2} (R\cdot V_\omega)^2||_{H^2(\R^2)}
   \lesssim ||V_\omega||_{H^{3+k}(\R^2)}^2,
   \label{}
  \end{equation}
  which implies \eqref{g4est1}.

  A calculation similar to the one in the proof of \eqref{g3wt}
  shows that in order to bound $x g_2$, we need to control
  the time integral of:
    \begin{equation}
   s ||\Lambda^{\iota} (R\cdot V_\omega)^2||_{L^2}
   + ||\Lambda^{-1/2+\iota} (R\cdot V_\omega)^2||_{L^2}
   + ||\Lambda^{1/2+\iota} \big((\Lambda^{-1} V_\omega) R\cdot V_\omega\big)||_{L^2}
   + ||\Lambda^{1/2 + \iota}\big( (R\cdot (xV_\omega) (R\cdot V_\omega)\big)||_{L^2}
   \label{}
  \end{equation}
  Using \eqref{fracprod} and \eqref{fracint}
  as in the proof of the previous lemma, it is straightforward to bound
  each of these terms by $(1+s) ||(1 + |x|^2)^{1/2} V_\omega||_{H^2(\R^2)}$.
\end{proof}

\subsection{Estimates for $g_5$}

Recall that $g_5$ contains all terms of order three or higher which
involve $V_\omega$.
There are two such types of terms: the terms coming from the first line of
\eqref{rearrange},
and the terms of degree 2 and higher
from expanding the rescaled Dirichlet-to-Neumann map $G(h)$ in powers of
$h$ and inserting this into \eqref{rearrange}.
In either case, the vorticity enters at most quadratically.
We illustrate how to handle the term corresponding to the first term on
the right-hand side of \eqref{rearrange}, which is:
\begin{equation}
 R_1 = -\int_0^t e^{i(s-t)\Lambda^{1/2}}
 |V_\omega \cdot \Lambda^{-1/2} \nabla u|^2.
 \label{}
\end{equation}
Using Sobolev embedding and e.g. the Hormander-Mikhlin multiplier theorem,
it is straightforward
to estimate:
\begin{equation}
 ||\nabla^k R_1||_{L^\infty(\R^2)}
 \lesssim ||R_1||_{H^{k+2}(\R^2)}
 \lesssim \int_0^t ||V_\omega(s)||^2_{W^{1+\epsilon,k+2}(\R^2)}
 ||u(s)||_{W^{(1+\epsilon)', k+3}(\R^2)}^2,
 \label{}
\end{equation}
for arbitrary $\epsilon > 0$. Using the interpolation inequality
\eqref{interpolationgms} and Young's inequality
$|ab| \lesssim |a|^p + |b|^q$ for $1/p + 1/q = 1$, this shows that:
\begin{equation}
 ||\nabla^k R_1||_{L^\infty(\R^2)}
 \lesssim \Big(\frac{\ve_0^2}{(1 + t)^{\sigma'}}\Big)
 \Big( \ve_1^2 (1+t)^{2+2\delta}\Big)
 \lesssim \ve_0^2 \frac{1}{1+t} + \ve_1^2 (1+t)^{3+3\delta}
 \label{}
\end{equation}
where $\sigma'  \ll 1$.
To estimate the terms coming
from the expansion of $G(h)$, one can argue as above, but using
additionally the estimates from
Appendix F of \cite{Germain2012}.

The estimates for $\Lambda^{\iota} x g_5$ are similar to the above
and the estimates we have already proved.
\section{Estimates for the dispersive terms}
\label{realdispsec}

In this section we bound the term $g_1$ defined in \eqref{duha2}.
We can actually proceed nearly exactly as in \cite{Germain2012} to
handle these terms. The only differences here are that (1)
after performing the normal
forms transformation (integration by parts in time), there are
additional terms involving the vorticity that need to be bounded
and (2)
we
want to control $||\Lambda^{\iota} xg_1||_{L^2(\R^2)}$ instead
of $||x g_1||_{L^2(\R^2)}$.

Recall the definitions of the bilinear, trilinear
and higher-order terms $B(u), T(u), R(u)$ from Proposition \ref{formulation}.
As in \cite{Germain2012}, after integrating by parts in time,
$B(u)$ can be written as a sum of terms whose Fourier transforms
are given by:
\begin{multline}
  \int_{\R^2} \mu(\xi, \eta) e^{it\varphi_{\alpha\beta}(\xi,\eta)}
  \hat{f}_{-\alpha}(t,\xi-\eta) \hat{f}_{-\beta}(t, \eta)\, d\eta
  - \int_0^t \int_{\R^2}\mu(\xi, \eta) e^{is\varphi_{\alpha\beta}(\xi,\eta)}
  \pa_s \Big(\hat{f}_{-\alpha}(s,\xi-\eta) \hat{f}_{-\beta}(s, \eta)\Big)\, d\eta
  \\ = A_1 + A_2
 \label{}
\end{multline}
where $\alpha, \beta \in \{+, -\}$,
$\varphi_{\pm\pm} = |\xi|^{1/2} \pm |\xi-\eta|^{1/2} \pm |\eta|^{1/2}$,
$f_+ = f, f_- = \overline{f}$ with $f = e^{it\Lambda^{1/2}} u$,
and $\mu$ is a bilinear multiplier which
is in the class $\B_1$, defined in Appendix C of \cite{Germain2012}.
The first term here can be estimated exactly as in Section 5 of
\cite{Germain2012} which gives:
\begin{equation}
 ||\nabla^k e^{-it\Lambda^{1/2}} \F^{-1}A_1||_{L^\infty(\R^2)}
 \lesssim \frac{\ve_0^2}{1+ t} \label{}
\end{equation}
To control $A_2$, it is enough to consider the case
that $\pa_s$ falls on the second factor. Using the equation
\eqref{duha2}, this generates two types of
terms: those involving just $\hat{f}$ and those involving $V_\omega$. The
first type of term can be dealt with just as in \cite{Germain2012}.
There are a large number of terms involving $V_\omega$ however they
can all be dealt with similarly to the estimates from the previous section.
This is because none of the above estimates involve any special
cancellations are are just performed by applying Sobolev embedding,
Holder's inequality and various simple facts from Harmonic analysis.
We just need to use Theorem C.1 from \cite{Germain2012} in place of
Holder's inequality.
For example, the term coming from
$g_2$ in \eqref{duha2} is:
\begin{equation}
 A_3 = \int_{0}^t \int_{\R^2} e^{is\varphi_{\alpha\beta}(\xi,\eta)}
 \mu(\xi,\eta) \frac{\eta}{|\eta|}
 \hat{f}_{-\alpha}(s,\xi-\eta)
 e^{-\beta is |\eta|^{1/2}}\widehat{V_\omega}(s,\eta)\, d\eta  ds,
 \label{}
\end{equation}
Applying Sobolev embedding $|| q ||_{L^\infty(\R^2)} \lesssim
|| q ||_{W^{1,p}(\R^2)}$ for $p > 10$, say,
Theorem C.1 from \cite{Germain2012} gives:
\begin{multline}
 ||\nabla^k e^{-it\Lambda^{1/2}} \F^{-1}A_3||_{L^\infty(\R^2)}
 \lesssim \int_0^t ||\nabla^k e^{-it\Lambda^{1/2}} \F^{-1} B_\mu(
 u, V_\omega)||_{W^{1,p}(\R^2)}\\
 \lesssim \int_0^t ||u(s)||_{W^{k+1, 2p}(\R^2)}
  ||V_\omega(s)||_{W^{k+1, 2p}(\R^2)}.
 \label{}
\end{multline}
Applying the bootstrap assumptions \eqref{bootstrap1}-\eqref{bootstrap2}
and the interpolation inequality \eqref{interpolationgms},
we get:
\begin{equation}
 ||\nabla^k e^{-it\Lambda^{1/2}} \F^{-1} A_3||_{L^\infty(\R^2)}
 \lesssim \int_0^t \frac{\ve_0}{(1 + s)^{1-\sigma}} \ve_1 (1 +s)^\delta\, ds
 \\
 \lesssim \ve_0 \ve_1(1+t)^{1+ \delta}.
 \label{}
\end{equation}

The estimates for $||\Lambda^{\iota} x g_1||_{L^2(\R^2)}$ can be proven
in a similar manner as above by following the outline in \cite{Germain2012}.
The only difference is that one needs to use the assumption
$||\Lambda^\iota (x e^{it\Lambda^{1/2}} u)||_{L^2} \leq (1 +t)^{\delta}$
in place of the assumption
$|| e^{it\Lambda^{1/2}} u)||_{L^2} \leq (1 +t)^{\delta}$ in
\cite{Germain2012}.
Summing up and noting that $V_\omega$ enters no more than quadratically
into any of the above terms, we get:
\begin{prop}
 If \eqref{bootstrap1}-\eqref{bootstrap3} hold for $\ve_0 \ll 1$
 and $\ve_1 \ll \ve_0$, then:
 \begin{equation}
  ||\nabla^k e^{-it\Lambda^{1/2}} g_1||_{L^\infty(\R^2)}
  \lesssim  \frac{\ve_0^2}{1+ t} + \ve_1 (1 + t)^{1+ 2\delta}
  \label{}
 \end{equation}
 for $k \leq N_1 + 4$, and
 \begin{equation}
  ||\Lambda^\iota (x g_1)||_{L^2(\R^2)}
  \lesssim \ve_0^2 (1 + t)^{\delta} +
  \ve_0\ve_1(1+t)^{1+2\delta}
  \label{}
 \end{equation}
\end{prop}

\section{Acknowledgments}
The author wishes to thank Hans Lindblad for suggesting this problem
and Pierre Germain for many helpful discussions and suggestions.

\appendix

\section{Interpolation and Sobolev inequalities}

In this section we will assume that $\pa \D_t$ is given
by the graph of a function, $\pa \D_t = \{ (x,h(t,x)), x \in \R^2\}$, and
further that we have a bound for the second fundamental form and
injectivity radius of $\pa \D_t$, as well as a bound for $|\nabla h|$:
\begin{equation}
 |\theta| + \frac{1}{\iota_0} + |\nabla h| \leq K.
 \label{}
\end{equation}
Note that $\theta \sim \nabla^2 h$. We then have the following Sobolev
inequalities:

\begin{lemma}
  If
$||u||_{L^6(\D_t)} + ||\na u||_{L^2(\D_t)} < \infty$,
then:
\begin{equation}
 || u||_{L^6(\D_t)} \leq C (K) ||\na u||_{L^2(\D_t)}.
 \label{sob1}
\end{equation}

If $u \in W^{k,p}(\D_t)$ then for $k > \frac{3}{p}$:
\begin{align}
 ||u||_{L^\infty(\D_t)} \leq C(K) ||u||_{W^{k,p}(\D_t)},
 \label{appsobint}
\end{align}
and if $u \in W^{k,p}(\pa \D_t)$, then for
$k > \frac{2}{p}$:
\begin{align}
 ||u||_{L^\infty(\pa \D_t)} \leq C(K) ||u||_{W^{k,p}(\pa \D_t)},
 \label{appsobbdy}
\end{align}
\end{lemma}
These estimates all follow from the estimates in the appendix of
\cite{Christodoulou2000}. The estimates
there are all stated for the case of a bounded domain but it is
clear that the proof goes through for an unbounded domain.

We will also need interpolation estimates on $\pa \D_t$ and
$\D_t$:

\begin{lemma}
  Let $2 \leq p \leq s \leq q \leq \infty$ and $0\leq k \leq m$. Suppose that:
  \begin{equation}
   \frac{m}{s} = \frac{k}{p} + \frac{m-k}{q}
   \label{}
  \end{equation}
  If $\alpha$ is a $(0,r)$ tensor then with $a = \frac{k}{m}$,
 \begin{align}
  ||\nabla^k \alpha||_{L^s(\pa\D_t)}
  \leq C ||\alpha||^{1-a}_{L^q(\pa \D_t)}
  ||\nabla^m \alpha||^a_{L^p(\pa \D_t)}
  \label{interp1}
 \end{align}
 and if $|\theta| + \frac{1}{\iota_0} \leq K$, then:
 \begin{align}
  \sum_{j = 0}^k ||\na^j \alpha||_{L^s(\D_t)}
  \leq C(K) ||\alpha||^{1-a}_{L^q(\D_t)}
  \bigg( \sum_{j = 0}^m ||\na^j \alpha||_{L^p(\D_t)}\bigg)^{a}.
  \label{interp2}
 \end{align}
\end{lemma}

Finally, we will use the following interpolation inequality
which is Lemma 5.1 in \cite{Germain2012}:
\begin{lemma}
  \label{interpolationgms}
 If $2 \leq p \leq \infty, k \leq N_0 + \frac{2}{p} -1$, then:
 \begin{equation}
  ||\nabla^k u||_{L^p(\R^2)}
  \lesssim (1+t)^{-1 + \frac{2}{p} + \sigma} \big((1 + t) ||u(t)||_{W^{4,\infty}(\R^2)}
  + (1 + t)^{-\delta} ||u(t)||_{H^{N_0}(\R^2)}\big),
  \label{}
 \end{equation}
 where $\sigma = \sigma(k, p, N_0,\delta) = \frac{k}{N_0 + \frac{2}{p} -1}(\delta
 - \frac{2}{p} + 1)$.
\end{lemma}

\section{Schauder estimates}

The following result is well-known (see e.g. Theorem 7.3 in \cite{Agmon1959}):
\begin{prop}
  If $f = F$ on $\pa \D_t$ and $\pa \D_t$ is given by the graph
  of $h :\R^2 \to \R$, then for $k \geq 2$:
 \begin{equation}
  ||f||_{C^{k,\alpha}(\D_t)}
  \leq C(||h||_{C^{k,\alpha}(\R^2)}) \big( ||\Delta f||_{C^{k-2,\alpha}(\D_t)}
  + ||F||_{C^{k-2,\alpha}(\pa \D_t)}
  + ||f||_{L^\infty(\D_t)}\big).
  \label{schauder}
 \end{equation}
\end{prop}

We will also need standard $L^p$ estimates
 (see e.g. Theorem 15.2 in \cite{Agmon1959}):
\begin{prop}
  If $f \in W^{2,p}(\D_t)$, $f = F$ on $\pa \D_t$
  and
  $\pa \D_t$ is given by the graph
  of $h :\R^2 \to \R$, , then:
  \begin{equation}
   ||f||_{W^{k,p}(\D_t)}
   \leq C( ||h||_{C^k(\R^2)} ) \big( ||\Delta f||_{W^{k-2,p}(\D_t)}
   + ||F||_{W^{k - 1/p}(\pa \D_t)} + ||f||_{L^p(\D_t)}\big).
   \label{lpschauder}
  \end{equation}
\end{prop}

\section{Estimates from harmonic analysis}

We collect a few results that we will use frequently.

\begin{lemma}
 \begin{itemize}
  \item If $1 < p \leq q < \infty$ and
 $\alpha = \frac{2}{p} - \frac{2}{q}$ then:
 \begin{align}
  ||\Lambda^{-\alpha} f||_{L^q} \lesssim ||f||_{L^p}
  \label{fracint}
 \end{align}
 \item If $1 < p < \infty$ and $s \geq 0$, then for any
 $1 < p_1,p_2,q_1,q_2 < \infty$ with $1/p_1 + 1/p_2 = 1/q_1 + 1/q_2 = 1/p$,
\begin{align}
 ||\Lambda^s (fg) ||_p \lesssim ||\Lambda^s f||_{p_1}
 ||g||_{p_2} +
 ||f||_{q_1} ||\Lambda^s g||_{q_2}.
 \label{fracprod}
\end{align}
 \end{itemize}
\end{lemma}

The estimate \eqref{fracint} is known as the Hardy-Littlewood fractional
integration lemma; for a proof, see \cite{Tao2006}. For a proof of \eqref{fracprod},
see \cite{Christ1991}.

We will also use the following estimate for the Dirichlet-to-Neumann map,
which is Proposition 2.2 from \cite{Germain2012}. As mentioned there, this is not
optimal (both in terms of the regularity assumed of $h$ and the
number of derivatives of $\varphi$ on the right-hand side) but this
will suffice for our purposes.
\begin{prop}
 If $\varphi: \pa\Omega \to \R$ where $\pa \Omega$ is the graph of a function
 $h$ with $h \in W^{4,\infty}(\R^2)$ then:
 \begin{equation}
  ||\N \varphi||_{W^{2,\infty}(\R^2)} \lesssim
  ||\nabla \varphi||_{W^{3,\infty}(\R^2)} +
  ||\Lambda^{1/2} \varphi||_{L^\infty(\R^2)},
  \label{dnmap}
 \end{equation}
 with implicit constant depending on $||h||_{W^{4,\infty}(\R^2)}$.
\end{prop}

\section{Elliptic systems}
\label{ellsyssec}

We follow the approach of \cite{Cheng2017} and \cite{Choi2013a}. First,
we define the space $Y$ to be closure of $C^\infty(\D_t)$ with respect to the
norm:
\begin{equation}
 ||u||_Y \equiv ||u||_{L^6(\D_t)} + ||\na u ||_{L^2(\D_t)}.
 \label{}
\end{equation}
We note that by the Sobolev inequality \eqref{sob1},
$Y$ is actually a Hilbert space with inner product:
\begin{equation}
 (u, v)_Y \equiv \int_{\D_t} \na u \cdot \na v.
 \label{}
\end{equation}

The goal of this section is to construct a solution
$\beta$ to the system:
\begin{align}
 \div \beta & = 0,&&\textrm{ in } \D_t,\label{div1}\\
 \curl \beta &= \alpha, && \textrm{ in } \D_t\label{div2},\\
 \beta \cdot N &= 0 && \textrm{ on } \pa \D_t,
 \label{div3}
\end{align}
where $\alpha\in L^{6/5}(\D_t)$.
Suppose for the moment that the following system has a unique weak
solution $\beta'$:
\begin{align}
 \Delta \beta' &= \alpha && \textrm{ in } \D_t,\label{betap1}\\
 \gamma^i_j \beta_i' &=0 && \textrm{ on } \pa \D_t,\label{betap2}\\
 \na_N \beta_N' &= - H \beta_N', && \textrm{ on } \pa \D_t,
 \label{betap3}
\end{align}
where $\na_N =N^j\na_j$, $H$ is the mean curvature of $\pa \D_t,
H = \tr \theta$ and $\beta_N' = N^i\beta_i'$. We now recall that by the
definition of the second fundamental form we have
$\div \beta'|_{\pa \D_t} =
\tr \na \beta'|_{\pa \D_t} = \div_{\pa \D_t} (\Pi \beta') +  H \beta_N$.
Taking the divergence of \eqref{betap1} and applying this formula
shows that $\beta'$ satisfies:
\begin{align}
 \Delta \div \beta' &= 0 && \textrm{ in } \D_t,\\
 \div \beta' & =0 && \textrm{ on } \pa \D_t,
 \label{}
\end{align}
so that $\div \beta' = 0$
in $\D_t$. In particular this implies that
$\Delta \beta' = \curl^2 \beta'$. If we then set $\beta = \curl \beta'$,
it follows that $\beta$ satisfies \eqref{div1} and \eqref{div2}. To
see that $\beta$ satisfies \eqref{div3}, we just note that $N\cdot \curl \beta$
only inolves tangential derivatives of $\gamma\cdot \beta'$ and thus
this vanishes if \eqref{betap2} holds. We also remark that this choice of $\beta$
is actually unique; if $\beta_1, \beta_2$ satisfy \eqref{div1}-\eqref{div3}
it follows that $\beta_1 - \beta_2 = \na \phi$ for some harmonic function $\phi$
which satisfies a Neumann problem with zero boundary data and is thus a constant.

We now prove that \eqref{betap1}-\eqref{betap3} has a unique weak solution:
\begin{prop}
  \label{divcurllem}
  Let $\alpha \in L^{6/5}(\D_t)$ and suppose that $H$, the mean curvature,
  satisfies:
  \begin{equation}
   ||H||_{L^3(\pa \D_t)} + ||\nabla H||_{L^{3/2}(\pa \D_t)}
   \ll 1.
   \label{smallcurv}
  \end{equation}
 Then the problem
  \eqref{betap1}-\eqref{betap3} has a unique solution $\beta' \in H^1(\D_t)$.
  Furthermore, with $\beta = \curl \beta'$, under the above hypotheses we have:
  \begin{equation}
   |\beta (z)| \lesssim \frac{1}{(1 + |z|)^2} \int_{\D_t}
   (1 + |z'|) |\alpha(z')|\, dz'
   \label{decay}
  \end{equation}
\end{prop}
\begin{proof}
  We let $C^\infty_{\textrm{tan}}(\D_t)$ denote the collection of smooth
  one-forms $\alpha$ on $\D_t$ so that $\gamma_i^j\alpha_j$ is compactly supported
  in $\D_t$, and
  we let $Y_0$ denote the closure of $C^\infty_{\textrm{tan}}(\D_t)$
  with respect to the norm $||u||_Y = ||u||_{L^6(\D_t)} + ||\na u||_{L^2(\D_t)}$.
  We define the bilinear form:
 \begin{equation}
  B[u, \varphi] \equiv \int_{\D_t} \delta^{ik}\delta^{j\ell}
   (\na_i u_j)( \na_k \varphi_\ell) + \int_{\pa \D_t} H u_N \varphi_N,
  \label{}
 \end{equation}
 for $u, \varphi \in Y_0$ and with $w_N = N^iw_i$. We want to find $u \in Y_0$ so that:
 \begin{equation}
  B[u,\varphi] = \int_{\D_t} \delta^{ij} \alpha_i\varphi_j,
  \label{}
 \end{equation}
 for all $\varphi \in Y_0$. The map $\varphi \mapsto \int_{\D_t} \alpha \cdot
 \varphi$ is a continuous linear map on $Y$ since $\alpha \in L^{6/5}(\D_t)$,
 and so by the Lax-Milgram theorem if suffices to prove that $B$ is bounded
 and coercive.

 Fix a smooth
 cutoff function $\chi = \chi(x,y)$ so that $\chi \equiv 1$ when $|y - h(x)| \leq
 \rho$ and $\chi \equiv 0$ when $|y - h(x)| \geq 2\rho$ for some fixed
 $\rho > 0$.
Let $\tilde{H} = \chi H$, and note that by Stokes' theorem we have:
 \begin{multline}
  ||\sqrt{\tilde{H}} w||_{L^2(\pa \D_t)}^2 =
  \int_{\D_t} \na_k \Big( N^k \tilde{H} |w|^2\Big)\, dxdy
  = \int_{\D_t} \na_k \Big(N^k \tilde{H}\Big) |w|^2\,dxdy
  + 2\int_{\D_t} \tilde{H} N^k \na_k w\cdot w\, dxdy
  \\
  \lesssim ||\nabla \tilde{H}||_{L^{3/2}(\D_t)}
  || w||_{L^6(\D_t)}^2 +
  ||\tilde{H}||_{L^3(\D_t)} ||\nabla w||_{L^2(\D_t)} ||w||_{L^6(\D_t)}
  \lesssim \epsilon^* || w ||_{Y}^2,
  \label{}
 \end{multline}
 where we used that
  $||\nabla \tilde{H}||_{L^{3/2}(\D_t)} \lesssim
 ||\nabla H||_{L^{3/2}(\D_t)}$ and $||\tilde{H}||_{L^3(\D_t)} \lesssim
 ||H||_{L^3(\D_t)}$. In particular this shows that the bilinear form
 $B$ is bounded on $Y$ and also, for sufficiently small $\epsilon^*$,
 that it is coercive on $Y$.

 We now prove the decay estimate \eqref{decay}. For this, we will construct
 a Green's function $G$ for the problem \eqref{betap1}-\eqref{betap3}, following
 the approach of \cite{Hofmann2007} and \cite{Choi2013a}.
 We fix $\rho > 0$ and let $G_\rho
= G_\rho(z,z')$ denote the weak solution to the problem \eqref{betap1}-\eqref{betap3}
with
$\alpha_i = \frac{1}{|\rho|^3}\chi_{B_{\rho}(z')}(z)$, $i = 1,2,3$, where
$B_\rho(z')$ denotes the ball of radius $\rho$ centered at $z'$
and $\chi$ is the cutoff function supported on this ball.
Following the
argument in section 4 of \cite{Hofmann2007},
one can prove that
\begin{align}
 ||G_\rho(z,\cdot)||_{W^{1,p}(B_{d_z(z)})} \leq C(d_y),
 \label{}
\end{align}
for some $p$ with $p \in (1, 3/2)$ and where $d_z$ denotes the distance
from $z$ to $\pa \D_t$. The constant here depends on $||h||_{C^1(\R^2)}$.
Taking a diagonal subsequence, for each $z$ we get a function $G(z,\cdot)
\in W^{1,p}(B_{d_z}(z))$
with $G_\rho(z,\cdot) \to G(z,\cdot)$ weakly in $W^{1,p}(B_{d_z}(z))$.
We would like to conclude the following two estimates:
\begin{equation}
 |G(z,z')| \leq C|z-z'|^{-1}, \quad
 |\na_{z} G(z,z')| \leq C|z-z'|^{-2},
 \label{}
\end{equation}
where $C = C(||h||_{W^{1,\infty}(\R^2)})$.
These estimates follows
as in Section 5 of \cite{Choi2013a} and Theorem 3.13 in
\cite{Kang2010}, provided that the system \eqref{betap1}-\eqref{betap3} satisfies
the condition ``(LH)'' in \cite{Kang2010}. However this follows from
Corollary 4.9 there provided that the system \eqref{betap1}-\eqref{betap3}
is sufficiently close to a diagonal system. Since we are assuming that
$||h||_{W^{4,\infty}(\R^2)}$ is small, this follows after straightening
the boundary.
We can now prove \eqref{decay}. We can assume $|z| \geq 1$. If
$|z'| \leq \frac{1}{2}|z|$, then $|z-z'| \geq \frac{1}{2} |z|$ so that:
\begin{equation}
 \big|\int_{\D_t} \na_z G(z,z') \alpha(z')\, dz'\big|
 \lesssim \int_{\D_t} \frac{1}{|z-z'|^2} |\alpha(z')|\, dz'
 \frac{1}{|z|^2} ||\alpha||_{L^1(\D_t)}.
 \label{}
\end{equation}
When $|z'| \geq \frac{1}{2} |z|$, we instead estimate::
\begin{equation}
 \big|\int_{\D_t} \na_z G(z,z') \alpha(z')\, dz'\big|
 \lesssim \int_{\D_t} \frac{1}{|z'|^2} |\alpha(z-z')|\, dz'
 \lesssim \int_{\D_t} \frac{1}{|z|^2} |\alpha(z-z')|\, dz'
 \lesssim \frac{1}{|z|^2} ||\alpha||_{L^1(\D_t)},
 \label{}
\end{equation}
as required.
\end{proof}
\begin{remark}
For the applications we have in mind, the assumption $||H||_{L^3(\pa \D_t)} \ll 1$
is not a serious restriction because we will actually have $||H||_{L^3(\pa \D_t)}
\leq \ve_0$. On the other hand, the assumption $||\nabla H||_{L^{3/2}(\pa \D_t)}
\ll 1$ only holds until $t \sim \ve_0^{-1/\delta} \sim \ve_0^{-N}$.
This condition is already forced on us by essentially
the result of Proposition \ref{uenlem},
`and so this is not a serious restriction for our purposes.
We note that it may be possible to remove this assumption
by arguing as in section 5.2 of \cite{Cheng2017}, but the arguments there are
somewhat involved.
\end{remark}




\end{document}